\documentclass[USenglish]{amsart}
\usepackage[pagewise]{lineno}

\usepackage{amssymb,amsmath,amsthm,amscd,epsf,latexsym,verbatim,graphicx,amsfonts,hyperref,epstopdf,xcolor,wasysym}
\usepackage{cancel}
\input epsf.tex
\usepackage{enumerate}
\usepackage[utf8]{inputenc}
\usepackage[T1]{fontenc}
\usepackage{graphicx}
\usepackage{palatino, url, multicol}
\usepackage{subfig}
\usepackage[margin=1.5in]{geometry}                
\usepackage{tikz}
\usetikzlibrary{patterns}
\usepackage{mathtools}

\usepackage{marginnote}

\theoremstyle{plain}
\newtheorem{theorem}{Theorem}[section]

\newtheorem{lemma}[theorem]{Lemma}

\newtheorem{proposition}[theorem]{Proposition}
\newtheorem{corollary}[theorem]{Corollary}

\newtheorem{mainthm}{Theorem}

\theoremstyle{definition}

\newtheorem{remark}[theorem]{Remark}

\newtheorem{definition}[theorem]{Definition}

\newcommand{\teapot}{\Upsilon_2^{cp}}

\title{The shape of Thurston's Master Teapot }
\author{Harrison Bray, Diana Davis, Kathryn Lindsey and Chenxi Wu}

\newif\ifdraft\drafttrue

\newcounter{comhar}
\def\0{{\mathbf 0}}

\newcommand{\concat}{ \hspace{.15em}{\cdot}\hspace{.15em}}

\begin{document}
\maketitle
\begin{abstract}
We establish basic geometric and topological properties of Thurston's Master Teapot and the Thurston set for superattracting unimodal continuous self-maps of intervals.  In particular, the Master Teapot is connected, contains the unit cylinder, and its intersection with a set $\mathbb{D} \times \{c\}$ grows monotonically with $c$.  We show that the Thurston set described above is not equal to the Thurston set for postcritically finite tent maps, and we provide an arithmetic explanation for why certain gaps appear in plots of finite approximations of the Thurston set. 
\end{abstract}

\section{Introduction}

In his last paper, unfinished at the time of his death, William Thurston
studied piecewise-linear maps of the unit interval \cite{thurston}. One
concept mentioned in this paper is an object that Thurston, in his 2012 course at Cornell University, affectionately called
the {\em Master Teapot},
 and which can be defined as follows. 
A unimodal endomorphism $f$ of a real interval is said to be \emph{critically periodic}
 if the critical point is a fixed point of some forward iterate of $f$, and is said to be \emph{postcritically finite} if the forward orbit of the critical point is a finite set. 
If $f$ is postcritically finite, it is easy to see that the orbit of the critical point determines a Markov partition of the interval.  
The Perron-Frobenius theorem then implies that the exponential of the topological entropy of $f$, $e^{h_{top}(f)}$, is a weak Perron number - i.e. a real, positive algebraic integer that is not less than the absolute value of any of its Galois conjugates - which we call the 
the {\em growth rate} of $f$ and denote by $\lambda(f)$.
Denote by $\mathcal{F}^{cp}$ the family of critically periodic unimodal continuous self-maps of compact real intervals. 
Then {\em Thurston's Master Teapot} is the set
\[
  \Upsilon_2^{cp}:=\overline{ \{(z,\lambda) \in \mathbb{C} \times \mathbb{R} \mid \lambda = \lambda(f) \textrm{ for some }f \in \mathcal{F}^{cp}, z \textrm{ is a Galois conjugate of } \lambda \}}.
\]

An application of $\Upsilon_2^{cp}$ is that it can be used as a necessary condition for a weak Perron number to be the growth rate of a critically periodic unimodal map: $\beta$ being such a number would imply that for each of its Galois conjugates $z$, $(z, \beta)\in\Upsilon_2^{cp}$.  Studying the Master Teapot may also inform the open question of completely classifying the set of dilatations of pseudo-Anosov surface diffeomorphisms, which may be thought of as two-dimensional analogues of uniformly expanding interval self-maps.  We call the image of the projection of the Master Teapot to $\mathbb{C}$ the \emph{Thurston set}; this set, which we discuss later, has been the subject of several recent works (e.g. \cite{ckw, TiozzoGaloisConjugates, thompson}).  Another motivation for studying these sets is that the part of the Thurston set inside the unit disk may be viewed as an analogue of the Mandelbrot set -- while the Mandelbrot set may be defined as the set of parameters $c \in \mathbb{C}$ such that $0$ belongs to the filled Julia set of the polynomial 
$z \mapsto z^2+c$, 
the part of the Thurston set inside the unit disk coincides with the set of parameters $z \in \mathbb{D}$ for which $0$ belongs to the limit set of the iterated function system generated by the maps $x \mapsto zx+1, x \mapsto zx-1$. Furthermore, a forthcoming article by the last two authors will show that each horizontal slice of the Master Teapot is an analogue of the Mandelbrot set.  This topic is also connected to the theory of ``core entropy,'' which has been the subject of numerous recent works (see, e.g. \cite{TiozzoTopologicalEntropy,TiozzoContinuity,GaoYanTiozzo,DegreeDInvariantLaminations}), as the restriction of a real quadratic polynomial to its Hubbard tree is a unimodal interval self-map.  




\begin{figure}[!h] 
\begin{center}
\includegraphics[width=0.7\linewidth,trim={16cm 11cm 10cm 3cm},clip]{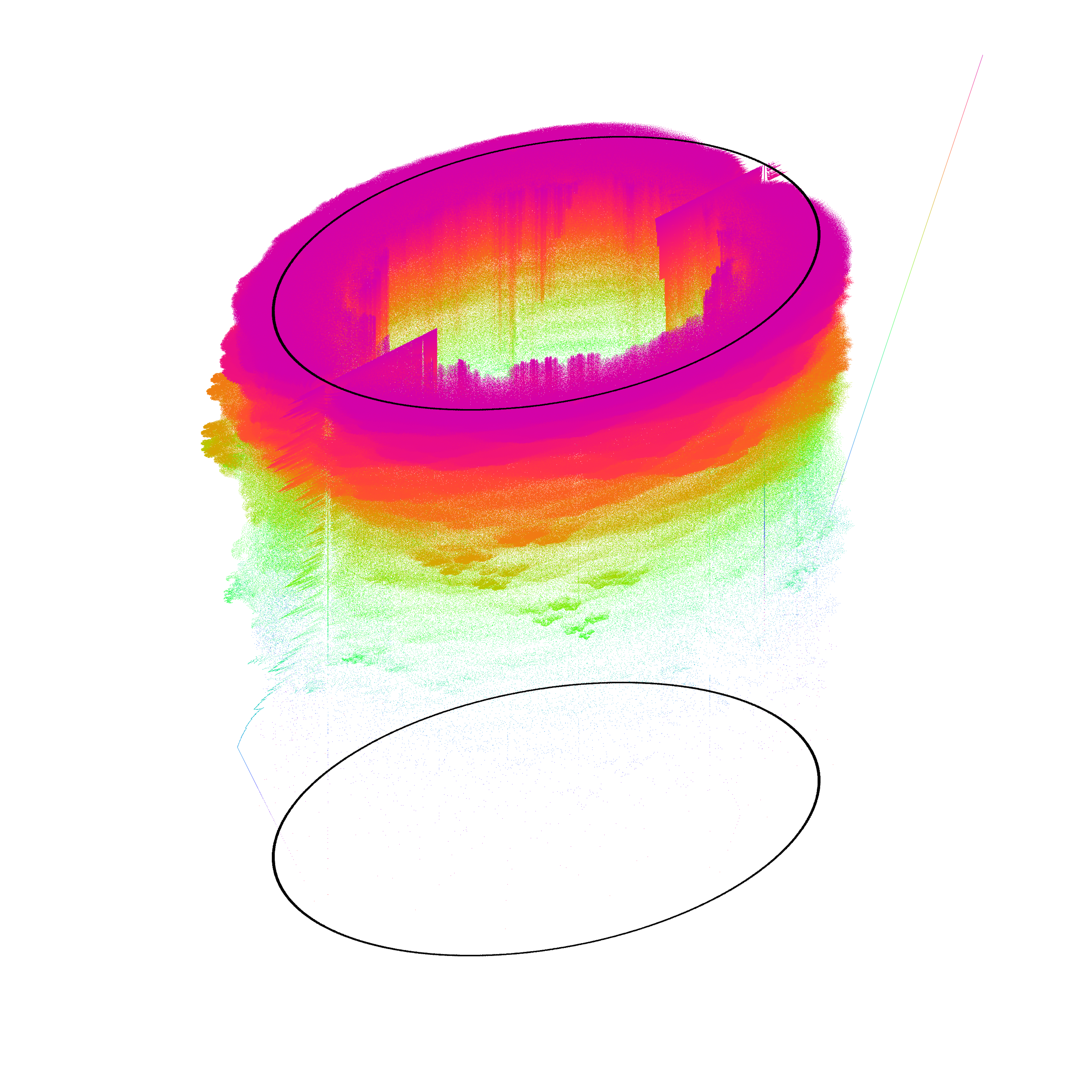}
\caption{A plot of a finite approximation of the Thurston Master Teapot, showing all points coming from maps in $\mathcal{F}^{cp}$ with critical period at most 23. The ``spout'' of the Teapot in the upper right corner is the line $\{(x,0,x) : x \in [1,2]  \subset \mathbb{R}^3 \simeq \mathbb{C} \times \mathbb{R}$, although the bottom of the spout is not visible in this finite approximation.   The ``handle'' of the Teapot in the lower left corner lies above the negative real axis.  The fact that the plot fades out towards the bottom of the Teapot is due to the finiteness of the approximation; considering maps with longer critical periods would give rise to more points near the bottom of the Teapot. The two black circles are the sets  $S^1 \times \{1\}$ and $S^1 \times \{2\}$, and the color gradients is according to the height of the points.
} 
\label{fig:gaps}
\end{center}
\end{figure}

Thurston describes the part of the Master Teapot $\Upsilon_2^{cp}$ \emph{outside} the unit cylinder as ``a network of very frizzy hairs, \ldots sometimes joining and splitting, but always transverse to the horizontal planes," \cite[Figure 7.7]{thurston} and the part \emph{inside} the unit cylinder as ``confined to (and dense in) closed sets that include the unit circle and increases [sic] monotonically with $\lambda$" \cite[Figure 7.8]{thurston}.
The first phenomenon is 
well-known (see e.g.  proof by Tiozzo \cite[proof of Theorem
1.3]{TiozzoGaloisConjugates}), but Thurston did not provide any further
explanation for the second. A main contribution in our paper is a proof of
the second phenomenon, which is that a point in the unit disc $\mathbb{D}$
which is on a horizontal slice of the Master Teapot persists as the height
of the slice increases. 
\footnote{To see this phenomenon in action -- that roots inside the unit cylinder persist, and also that roots outside the unit cylinder move continuously -- see our video \url{https://vimeo.com/259921275}.}

\begin{mainthm}[Persistence]
  \label{mainthm:closurepersistence}
  For any point $z\in\mathbb{C}$ in the open unit disk $\mathbb{D}$, if
  $(z,\beta)$ is in the Master Teapot, then every point above it up to
  height 2 is also in the Master Teapot. In other words,
  \[
    (z,\lambda) \in \Upsilon_2^{cp}\text{ implies }\{z \} \times
    [\lambda,2] \subset \Upsilon_2^{cp}.
  \]
\end{mainthm}

Two corollaries of this main theorem are the following:

\begin{mainthm}[Unit Cylinder] \label {mainthm:unitcylinder}
    The Master Teapot $\Upsilon_2^{cp}$ contains the unit cylinder $S^1 \times [1,2]$.
\end{mainthm}
Another equivalent way to state the Unit Cylinder Theorem
\ref{mainthm:unitcylinder} is that $S^1\times \{1\}$ is contained in the
Master Teapot, and the Persistence Theorem \ref{mainthm:closurepersistence}
holds on the closed unit cylinder.
  
\begin{mainthm}[Connectedness] \label{mainthm:connected}
	The Master Teapot $\Upsilon_2^{cp}$ is connected.  Furthermore, \mbox{$\Upsilon_2^{cp} \cap (\overline{\mathbb{D}} \times [1,2])$} is path-connected.  
\end{mainthm}

We also proved a number of results that are not logically dependent on the main theorem above and concern other sets related to the Master Teapot.
Let the {\em Thurston set}, which we denote $\Omega_2^{cp}$, be the
projection of Thurston's Master Teapot onto $\mathbb{C}$:
\[\Omega_2^{cp} :=\overline{ \{z \in \mathbb{C} \mid \lambda = \lambda(f) \textrm{ for some }f \in \mathcal{F}^{cp}, z \textrm{ is a Galois conjugate of } \lambda \}}.
\]
In other words, the Thurston set $\Omega_2^{cp}$ is the closure of the set
containing all Galois conjugates of growth rates of unimodal maps which are
critically periodic. 

A heretofore mysterious feature of plots of finite approximations of the Thurston set, formed by
bounding the length of the postcritical orbits, was the appearance of visible ``gaps'' or holes at
fourth roots of unity, sixth roots of unity, and certain other algebraic numbers. 
\begin{figure}[!hb]
\begin{center}
\includegraphics[width=\linewidth,trim={15cm 19cm 10cm 19cm},clip]{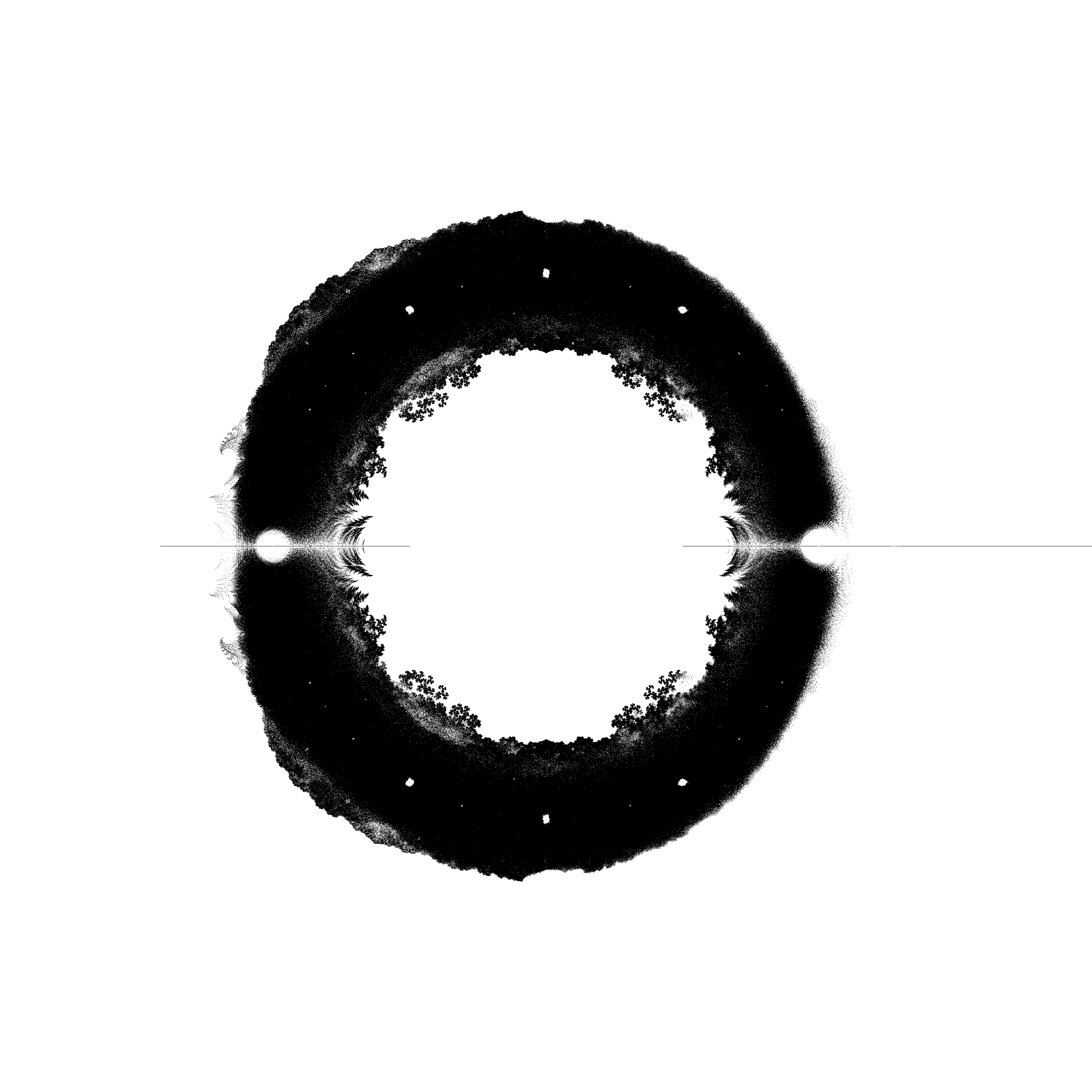}
\caption{An approximation of the Thurston set, $\Omega_2^{cp}$, consisting
of the roots of all minimal polynomials associated to postcritically finite
tent maps for which the post-critical period is at most 25}
\label{fig:thurston_set}
\end{center}
\end{figure}

The gaps on the unit circle get filled in as the length of the
postcritical orbits approaches infinity \cite[Proposition 6.1]{TiozzoGaloisConjugates}.  It is
known, however, that $\Omega_2^{cp}\color{black} \cap \mathbb{D}$ does have a hole other than the large central hole around the origin \cite{ckw}. The Gap Theorem~\ref{mainthm:gaps} provides an arithmetic explanation for these visible gaps in finite approximations of $\Omega_2^{cp}$.  

\begin{mainthm}[Gaps]
  \label{mainthm:gaps}
  For $n \in \mathbb{N}$, let  $\omega_n$ denote the set of Galois conjugates of growth rates of unimodal critically periodic maps with postcritical length at most $n$.   Let $R$ be one of the rings 
  $\mathbb{Z}[\sqrt{-D}]$ or $\mathbb{Z}[\frac{1+\sqrt{-3}}{2}]$ 
  for $D=1,2,$ or $5$, 
  and set $c = \inf\{|z| : z \in R, z \not = 0\}$.  
	Then for any $x \in R$, 
	$$B_{r(x)}(x)  \cap \omega_n \subset \{x\},$$
where 
$$r(x) = 
\begin{cases}
\min \left \{ \frac{c}{(2n^2 + 3n+1) |x|^n e},  \frac{1}{n+1} \right \} & \textrm{ if } |x| \geq 1,\\
\min \left\{ \frac{c}{(2n^2+3n+1) e}, \frac{1}{n+1} \right \}  & \textrm{ if } |x| \leq 1.\\
\end{cases}$$
\end{mainthm}

Let $\mathcal{F}^{pcf}$ be the family of unimodal postcritically finite self-maps of real intervals.  
We define the {\em postcritically finite Thurston set}, $\Omega_2^{pcf}$, as
\[\Omega_2^{pcf} :=\overline{ \{z \in \mathbb{C} \mid \lambda = \lambda(f) \textrm{ for some }f \in \mathcal{F}^{pcf}, z \textrm{ is a Galois conjugate of } \lambda \}}.
\]
In other words, the postcritically finite Thurston set $\Omega_2^{pcf}$ is
the closure of the set containing all Galois conjugates of growth rates of
unimodal maps which are postcritically finite. 

We proved that:

\begin{mainthm}[Two Thurston Sets] \label{mainthm:prepernotequal}
The Thurston set $\Omega_2^{cp}$ and the postcritically finite Thurston set $\Omega_2^{pcf}$ are not equal.
\end{mainthm}

\begin{figure}[h!]
  \centering
  \includegraphics[width=\linewidth,trim={15cm 19cm 10cm 19cm},clip]{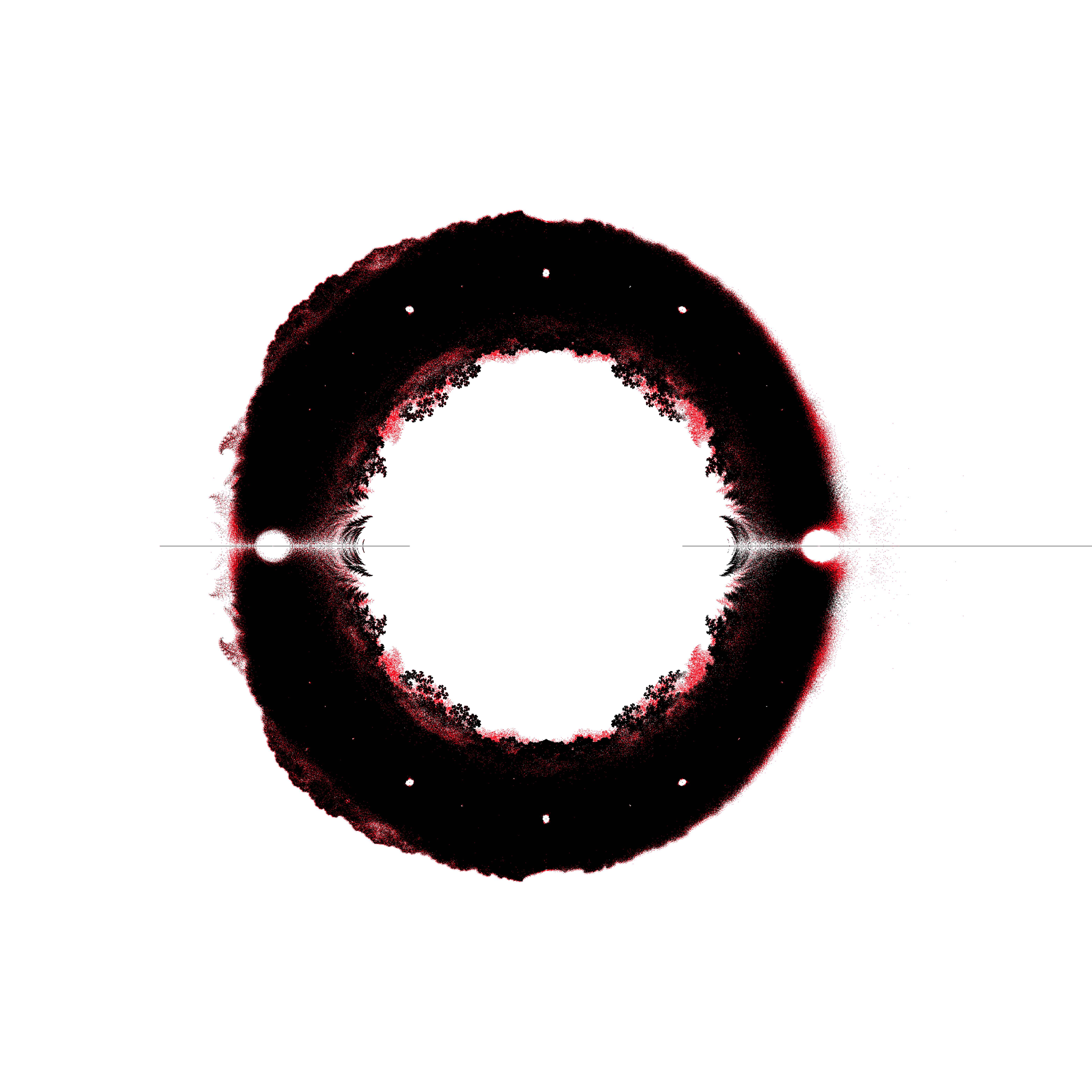}
  \caption{Here is an overlay of finite approximations of 
    the two Thurston sets: the black image contains points in
    $\Omega_2^{cp}$ corresponding to period length up to $25$, 
    and the red image contains points in
    $\Omega_2^{pcf}$ corresponding to preperiod plus period up to $22$.
    The set $\Omega_2^{pcf}$ is shown on its own in Figure \ref{fig:preperiodic}.}
  \label{fig:overlay}
\end{figure}

The caption of Thurston's image \cite[Figure 1.1]{thurston} states that the image shows the roots of
the defining polynomials for "a sample of about $10^7$ postcritically finite quadratic maps of
the interval with postcritical orbit of length $\leq 80$." We suspect, based on visual comparison of plots, that Thurston's image shows
only roots of critically periodic tent maps, i.e. shows $\Omega_2^{cp}$ and not $\Omega_2^{pcf}$ (c.f. Figure \ref{fig:overlay}).

At the moment, we do not have a good understanding of the shape of the postcritically finite Thurston set $\Omega_2^{pcf}$ and the 
analogously defined ``teapot'' $\Upsilon_2^{pcf}$; for example, we do not know if they exhibit persistence (as in the Persistence Theorem~\ref{mainthm:closurepersistence}) or connectivity (as in the Connectedness Theorem~\ref{mainthm:connected}).

  \subsection{Perspectives on the Thurston set}


 Our main tool for the study of $\Omega_2^{cp}$, and unimodal maps on intervals in general, is the Milnor-Thurston kneading theory. The Milnor-Thurston kneading theory \cite{MilnorThurston} (also cf. \cite{Guckenheimer}) provides the connection between general unimodal maps, real quadratic maps and subshifts in certain symbolic dynamical systems via entropy-preserving semi-conjugacies, and connect them to the study of infinite power series with prescribed coefficients called kneading determinants. As a result, there are numerous characterizations of $\Omega_2^{cp}$ from different points of view, and our results build (directly or indirectly) on a long history of research in each of these areas.
  
 1. \emph{Polynomials and power series with prescribed coefficients} An alternative way to describe the kneading determinant and kneading polynomials, which predates Milnor-Thurston kneading theory, is $\beta$-expansions and Parry polynomials, which were first introduced in \cite{parry60} for maps of the form $x \mapsto \beta x \mod{1}$ and later extended to a larger class of piecewise linear interval self-maps (e.g. \cite{gora, ItoSadahiro, DombekMP, Steiner,  IntermediateBetaShifts}). Solomyak \cite{solomyak} used Parry polynomial to study the closure of the Galois conjugates of $\beta$ such that $x \mapsto \beta x \mod{1}$ has finite critical orbit, Thompson \cite{thompson} used it to study a set that contains the Thurston set, and the distribution of roots of Parry polynomials was studied in \cite{VGJL,VGJL2}.
 More generally, there is a large body of literature that investigating the roots of polynomials and power series with all coefficients in a prescribed set (see, for example, \cite{OdlyzkoPoonen, BeaucoupEtAl, BorweinEtAlLittlewoodType, Konyagin, ShmerkinSolomyak, BorweinErdelyiLittmann}). The polynomials most closely related to the Thurston set are perhaps Littlewood, Newman and Borwein polynomials, polynomials whose coefficients belong to the sets $\{\pm 1\}$, $\{0,1\}$ and $\{-1,0,+1\}$ respectively.

 2. \emph{Complex dynamics.} Since the study of unimodal maps can be reduced to the study of real quadratic maps, the study of entropies of critically periodic unimodal maps is reduced to the study of core entropy on superattracting parameters on the real slice of the Mandelbrot set. The study of the core entropy on the Mandelbrot set is a rich subject, cf. \cite{douady1984etude, poirier2009critical, li2007monotonicity, meerkamp2013hausdorff, thurston2016rubber, TiozzoContinuity, TiozzoTopologicalEntropy}.

 3. \emph{Symbolic dynamics and Iterated function systems (IFS).} The kind of symbolic dynamical systems semiconjugate to a real quadratic map was described in \cite{MilnorThurston} via a combinatorial ``admissibility criteria''. Using this, Tiozzo \cite{TiozzoGaloisConjugates} proved that the Thurston set $\Omega_2^{cp}$ is connected, locally connected, and contains a uniform neighborhood of the unit circle. In particular, \cite{TiozzoGaloisConjugates} shows that a point $z$ with absolute value less than $1$ is in the Thuston set $\Omega_2^{cp}$ if and only if $0$ is in the limit set of the iterated function system generated by the two maps $x \mapsto zx+1$ and $x \mapsto zx -1$. This and some other related IFS are the focus of numerous works, including 
\cite{BarnsleyHarrington, BouschPaires, BouschConnexite, Bandt, SolomyakXu,
SoloymakLocalGeom, SolomyakAsymptotic}.  In \cite{ckw}, Calegari, Koch
and Walker used this and a related IFS to prove that the Thurston set has a hole, in addition to the obvious, large hole of radius $1/2$ centered at $0$. 

\subsection{Structure of the paper} \ 
A major consequence of the Milnor-Thurston theory is that unimodal maps on
intervals are semiconjugate to {\em tent maps} with the same entropy. This
tool is essential in our method of proof. 
 
{\bf  \S\ref{s:preliminaries}: Preliminaries} 
We define the \emph{$\beta$-itinerary} of a point under a tent map, 
\emph{Parry polynomials}, and give the \emph{admissibility criterion} for itineraries, which are key tools in our arguments.

{\bf \S\ref{sec:connections}:
Quadractic maps, iterated function systems, and renormalization
}
provides background on Milnor-Thurston kneading theory and reviews the
concept of \emph{renormalization}. 

{\bf  \S\ref{sec:auxiliary}: Dominant words} reviews the definition and
properties of dominant words from Tiozzo's work
\cite{TiozzoTopologicalEntropy}. 
 
%
 
 {\bf\S\ref{sec:persistence}: Persistence} proves the main theorem, Persistence Theorem~\ref{mainthm:closurepersistence}.

 {\bf\S\ref{sec:cylinder}: The unit cylinder and connectivity} shows that
 the Master Teapot is connected inside the unit cylinder, and uses this
 structure to prove the Unit Cylinder Theorem~\ref{mainthm:unitcylinder}
 and the Connectedness Theorem \ref{mainthm:connected}. 

 {\bf\S\ref{sec:gaps}: Gaps in the Thurston set} explains why there appear to be ``holes'' near primitive roots of unity in the finite approximations of the Thurston set. We show that these holes are associated to discrete subgroups, proving the Gap Theorem~\ref{mainthm:gaps}. 

{\bf  \S\ref{sec:preperiodic}: $\Omega_2^{cp}$ and $\Omega_2^{pcf}$ are not
equal} shows that the periodic and preperiodic Thurston sets are not equal,
proving the Two Thurston Sets Theorem~\ref{mainthm:prepernotequal}.

\subsection{Acknowledgements}
The authors gratefully acknowledge Giulio Tiozzo, Daniel Thompson, Sarah Koch, and Dylan Thurston for helpful conversations.  This work began at the AMS Mathematics Research Communities program in June 2017. The authors are immensely grateful to the MRC program for stimulating this collaboration, and to Daniel Thompson introducing us to this subject while at the MRC.  This material is based upon work supported by the National Science Foundation under Grant Number DMS 1641020. The first author was supported in part by NSF RTG grant 1045119.  The third author was supported in part by NSF DMS grants 1901247 and 1401133.

\section{Preliminaries}\label{s:preliminaries}


\begin{theorem} \cite[Theorem 7.4]{MilnorThurston}
  \label{thm:MTsemiconjugacy}
  Every continuous self-map $g$ of an interval with finitely many turning
  points and with $h_{top}(g) > 0$ is semi-conjugate to a uniform
  $\lambda$-expander $PL(g)$ with the same topological entropy $h_{top}(g) =
  \log \lambda$. If $g$ is postcritically finite, so is $PL(g)$.
\end{theorem} 

 Thus,
to understand Thurston's Master Teapot, it will suffice to study these more
rigid dynamical system. 

\subsection{Tent maps}


Denote the unit interval by $I=[0, 1]$. For fixed $\beta\in (1, 2]$, the {\em
tent map} of slope $\beta$ is the continuous, piecewise linear map
$f_\beta$ of the unit interval $I$ defined by:
 
\[
  f_\beta= \begin{cases} \beta x & x\leq {1\over\beta}\\ 2-\beta x &
    x>{1\over\beta}\end{cases}.
\]

For a continuous self-map $f$ of an interval with finitely many turning
points, the topological entropy $h(f)$ is equal to the following
limit:
\begin{equation} \label{eq:arclengthentropy}
  h(f) = \lim_{n \to \infty} \frac{1}{n} \log(\textrm{Var}(f^n)),
\end{equation} 
where $\textrm{Var}(f)$ denotes the total variation of $f$ \cite{MS}. Then
a straightfoward calculation confirms that for a tent map $f_\beta$, the
{\em growth rate}, which is the exponential of the topological entropy, is
equal to $\beta$. In other words, $e^{h(f_\beta)}=\beta$. 

Via Milnor-Thurston's entropy-preserving semi-conjugacy
\cite{MilnorThurston}, the critical point of a unimodal map is sent to the
unique preimage of $1$ under the associated tent map. Then a tent map
$f_\beta$ is said to be {\em postcritically finite} if the $f_\beta$-orbit
of $1$ is finite, and $f_\beta$ is said to be {\em critically periodic} if $1$ is a
periodic point for $f_\beta$. Now we have alternative interpretations of 
Thurston's Master Teapot, the Thurston set, and the postcritically finite
Thurston set: 
\begin{itemize}
  \item The Master Teapot $\Upsilon_2^{cp}$ is the closure of the
    set of all pairs $(z,\beta)$ in $\mathbb C\times\mathbb R$ for which
    $z$ is a Galois conjugate of $\beta$, and $\beta$ is 
    the growth rate of a critically periodic tent map;
  \item the Thurston set $\Omega_2^{cp}$ is the closure of the set of all
    Galois conjugates of growth rates of critically periodic tent maps;
  \item the postcritically finite Thurston set $\Omega_2^{pcf}$ is the
    closure of the set of all Galois conjugates of growth rates of postcritically
    finite tent maps. 
\end{itemize}

\subsection{Combinatorial itineraries}

The dynamics of the tent map $f_\beta$ can be represented by a relatively
simple Markov coding: there is a Markov partition of the unit interval into
two subintervals labeled 0 and 1, and we represent the $f_\beta$-orbit of
any point $x$ with an itinerary sequence whose $n$-th term is 0 or 1
depending on which subinterval contains the $n$th iterate of $x$.  We will
make this representation more precise later in this section, since we will
extensively use the Markov coding of a tent map in this work.  First, we
will define essential abstract data of sequences and words in the alphabet
$\{0,1\}$. 

\begin{definition}
  \label{def:words_strings_seqs} 
  We will use the term \emph{string} to refer to an ordered list of letters
  in some alphabet, and this list may be either finite or infinite. We
  adopt the convention that a \emph{word} is always a finite string, and a
  \emph{sequence} is always an infinite string. An itinerary is also
  assumed to be an infinite string. We often concatenate a word $w$ with
  the notation $w^n$, which is the word created from repeating $w$ exactly
  $n$ times. Similarly, $w^\infty$ is the sequence created by repeating
  $w$ infinitely many times. 
\end{definition}

\begin{definition}
  The sequence of \emph{signs} associated to a sequence $w=(w_1 w_2 \dots)
  \in \{0, 1\}^{\mathbb{N}}$ is the sequence $e_w:\mathbb{N} \to \{-1,+1\}$
  defined by 

  $$e_w(j) = \left\{
    \begin{array}[]{ll}
      +1 & \text{ if }w_j=0, \\
      -1 & \text{ if }w_j =1.
    \end{array}
    \right.
    $$
    The sequence of {\em cumulative signs} associated to a sequence $w=(w_1
    w_2 \dots)\in \{0, 1\}^{\mathbb{N}}$ is the sequence $s_w \colon
    \mathbb{N} \to  \{+1,-1\}$ defined by
    $s_w(1)=1$ and
    \begin{equation} \label{eq:cumulativesigndef}
      s_w(j+1)=\prod_{k=1}^{j} e_w(k)
    \end{equation}
    for $j\geq1$. In other words, the $(k+1)^{\text{st}}$ sign $s_w(k+1)$ is equal to 1 if
    and only if the sum of the first $k$ entries of the sequence $w$ is even. 
    If $w$ is a finite string, the {\em cumulative sign} of $w$ is defined
    as $\prod_k e_w(k)$.
%
    The sequence of \emph{digits} associated to a sequence $w=(w_1 w_2
    \dots)\in \{0, 1\}^{\mathbb{N}}$ is the sequence $d_w:\mathbb{N} \to
    \{0,2\}$ defined by $d_w(i) = 2w_i$.  
\end{definition}

\subsection{Ordering on the set of strings} \label{sec:admissibility}

\begin{definition}[Twisted lexicographic ordering] 
\label{def:twistedlexicographic} 
{\color{white} for formatting only}

\begin{enumerate}
 \item 
Define the ordering $\leq_E$ on the set of sequences in  $\{0,1\}^{\mathbb{N}}$ as follows.  
Given two distinct sequences
$w=(w_1w_2\dots)$ and $v=(v_1v_2\dots)$ in $\{0,1\}^{\mathbb{N}}$, define
$w <_E v$ if and only if at the first integer $n$ such that $w_n\neq v_n$, 
\[\left\{
\begin{array}{ll}
	w_n<v_n &\text{ if } s_w(n)=+1,\\
    w_{n}>v_{n} &\text{ if } s_w(n)=-1.
\end{array}
\right.
\]
Note that $s_w(n)=s_v(n)$ by definition since $n$ is the first index at
which the sequences $w$ and $v$ differ. 
\item Define the ordering $\leq_E$ on the set of words in the alphabet $\{0,1\}$ as follows.  Given two words $w$ and $v$, write 
$w <_E v$ if and only if $w^\infty <_E v^\infty$.  
\end{enumerate}
\end{definition}


\begin{remark}
  It is straightforward to check from the definition of twisted
  lexicographic ordering that if a word
  $a$ has positive cumulative sign, then for any strings $v,w$, we have
  $w<_E v$ if and only if $aw<_E av$. Similarly, if $a$ has negative
  cumulative sign, then $w<_E v$ if and only if $aw>_Eav$. 
  \label{rem:flippingsigns}
\end{remark}

Now we can define the concept of $\beta$-itinerary as below:

\begin{definition}[$\beta$-itinerary]
  Let $I_0^\beta=[0, 1/\beta]$, $I_1^\beta=[1/\beta, 1]$. The
  $\beta$-itinerary of the tent map $f_\beta$ is the sequence
  $w=(w_1w_2\dots)$ satisfying the following two conditions:
  \begin{enumerate}
  \item $f^n(1)\in I_{w_{n+1}}^\beta$.
  \item Among all the sequences satisfying the preceding condition
    (1), $w$ is
    the minimal such sequence under the twisted lexicographical ordering. 
  \end{enumerate}
\end{definition}

It is obvious that if $f_\beta$ is not critically periodic, there is a unique sequence satisfying condition (1) which has to be the $\beta$-itinerary. If $f_\beta$ is critically periodic, one can easily check that the $\beta$-itinerary can be equivalently defined explicitly as follows: if $f_\beta^k(1)=1$
  and $k$ is minimal, then this is the itinerary $w^\infty$ where $w$ has
  length $k$ and the last digit of $w$ is chosen such that $w$ has positive
  cumulative sign. From this observation one can see that this definition is consistent with the standard kneading theory definition of the itinerary of $1$ under $f_\beta$, which is the limit of the itineraries of $x_i$ under $f_\beta$, where $x_i\in [0, 1]$, $\lim_ix_i=1$, and the forward orbit of $x_i$ never hits any critical point. 

\subsection{Parry polynomials}

The definition below for a Parry polynomial is motivated by the concept of $\beta$-expansion.

\begin{definition}  \label{def:ParryPolynomial} 

  Let $w$ be a word in the alphabet $\{0,1\}$. Set $f_0^z(x)=zx$ and
  $f_1^z(x)=2-zx$.  Then the \emph{Parry polynomial} for $w$ is
  \begin{align}\label{eqn:itinerary_of_x}
    \begin{split}
    P_w (z) & := s_w(p+1)(f_{w_p}^z\circ f_{w_{p-1}}^z\circ \dots \circ
    f_{w_1}^z(1)-1) \\
    & = z^p-s_w(1)d_w(1)z^{p-1}-\cdots-s_w(p)d_w(p) - s_w(p+1) \\
            & = (z-1)(z^{p-1}+s_w(2)z^{p-2}+\dots+s_w(p)).
            \end{split}
  \end{align}
\end{definition}

When $f_\beta$ is critically periodic, the first line of equation
\eqref{eqn:itinerary_of_x} confirms that for any word $w$ for which
$w^\infty$ is a $\beta$-itinerary, we have that $\beta$ is a root of the
Parry polynomial $P_w$.  
Thus, the minimal polynomial for $\beta$ is a factor of $P_{w}$ for any word
$w$ such that $w^\infty$ is a $\beta$-itinerary.  
As a final observation, $P_w$ is also never irreducible over the integers, 
as it always has a
factor of $(z-1)$. At times it will be important for our arguments to
ensure that the Parry polynomial has only this one extra factor of $(z-1)$,
i.e. has exactly two irreducible factors. 

\subsection{Admissible itineraries}


  


Let $\sigma\colon \{0,1\}^\mathbb{N} \to \{0,1\}^\mathbb{N}$ be the
standard {\em shift map}, defined by $\sigma(w_1w_2w_3\dots)= (w_2w_3\dots)$.  

Milnor-Thurston developed a combinatorial criterion for a sequence in
$\{0,1\}^{\mathbb N}$ to be realized as an itinerary of the critical
value under a 
quadratic map from the family $g_c\colon x\mapsto x^2+c$, where  
$c$ is real. A quadratic map is given an itinerary in the same procedure as
for a tent map; partition the domain of the map into two intervals whose
intersection is the critical point, and the left interval receives a coding
value of 1 while the right interval receives a coding value of 0.


\begin{theorem} \cite[Theorem 12.1]{MilnorThurston}
\label{t:admissible}
A sequence $a=(a_n)$ in $\{0,1\}^{\mathbb N}$ is an itinerary of the
critical value of a quadratic map if and only if 
$\sigma^j(a) \leq_E a$ for all $j \in \mathbb{N}$. 
\end{theorem}

As a corollary of Milnor-Thurston's semi-conjugacy from Theorem
\ref{thm:MTsemiconjugacy}, a sequence in
$\{0,1\}^{\mathbb N}$ which is realizable as an itinerary of 
of 1 under a tent map is also realizable as an itinerary of the critical
value of a quadratic map, as 1 is the image of the critical value under
semi-conjugacy. 
Thus, Theorem \ref{t:admissible} introduces a necessary combinatorial
condition on 
$\beta$-itineraries which we call admissibility:

\begin{definition}[Admissibility] \label{def:admissible}
  A sequence $a=(a_1a_2\ldots)$ in the alphabet $\{0,1\}$ is {\em admissible}
  in the Milnor-Thurston sense if for all positive integers $j$, the
  shifted sequence satisfies the inequality
  \[
    \sigma^j(a)\leq_E a. 
  \]
  Then a word $w$ is admissible if and only if the sequence $w^\infty$ is
  admissible. 
\end{definition}


On the other hand, the converse is more subtle because the Milnor-Thurston
semi-conjugacy is not a true conjugacy, the reason being that a critically
periodic tent map is semi-conjugate to infinitely many quadratic maps with
different post-critical itineraries. However we do have the partial
converse which will be sufficient for our purposes:

\begin{proposition}\label{p:min}
  Let $w$ be a word in the alphabet $\{0,1\}$ with positive cumulative
  sign.  If $w$ is admissible 
  and the Parry polynomial associated
  with $w$ can be factored into $z-1$ and another irreducible factor, then
  $w^\infty$ is the $\beta$-itinerary
  for some $\beta\in(1,2]$. 
\end{proposition}

\begin{proof}
Theorem 12.1 of \cite{MilnorThurston} tells us that if $w^\infty$ is
admissible, it must be the itinerary of $c$ under some quadratic map $g_c:
x\mapsto x^2+c$ (here we let $I_1=(-\infty, 0]$, $I_0=[0, \infty)$).
  Because a quadratic map is unimodal, we can find some tent map $f_\beta$
  semi-conjugate to the quadratic map $g_c$. Suppose $(w')^\infty$ is the
  $\beta$-itinerary, and $w'$ has minimal length. 
  The proof of Lemma 12.2 in \cite{MilnorThurston} implies (which can also
  be checked by bookkeeping) that the 
  itinerary of the critical value $c$ under any quadratic map $g_c$ 
  which is semi-conjugate to $f_\beta$ must lie between $(w')^\infty$ and
  $(w'')^\infty$, where $w''$ has the same length as $w'$ and
  agrees with $w'$ except for the last letter. Hence, $(w')^\infty\le_E
  w^\infty\le_E (w'')^\infty$, which implies that the length of $w$ must be
  a multiple of the length of $w'$.
  Because $\beta$ is a root of the Parry
  polynomial associated with $w$, $w'$ and $w''$, the fact that the Parry
  polynomial of $w$ has only two irreducible factors implies that $w$ and
  $w'$ have the same length. Hence $w=w'$ or $w=w''$. The condition that $w$
  has positive cumulative sign precludes $w=w''$, so $w=w'$. 
\end{proof}

  
Note that the converse of Proposition \ref{p:min} is false; if we write 
the $\beta$-itinerary as $w^\infty$, this string 
is admissible in the Milnor-Thurston sense, 
but the Parry polynomial $P_w$ may have more than two irreducible factors,
even if $w$ is minimal length and has positive cumulative sign. 

For a critically periodic tent map $f_\beta$, we call the Parry polynomial
associated with the $\beta$-itinerary {\em the Parry
polynomial of $f_\beta$} and denote it by $P_\beta$.  In the case that
$f_{\beta}$ is postcritically finite, 
a similar procedure using the sum of a power series produces a polynomial
associated to a preperiodic word, and hence to the preperiodic 
$\beta$-itinerary.

\subsection{Irreducibility} 


To check that a Parry polynomial has only two irreducible factors, we will
use two lemmas from  \cite{TiozzoGaloisConjugates}, which are derived from
Eisenstein's criterion.

\begin{lemma}  \cite[Lemma 4.1]{TiozzoGaloisConjugates} \label{l:tiozzo4point1}
Let $d=2^n - 1$ with $n \geq 1$, and choose a sequence $\epsilon_0,\epsilon_1,\dots,\epsilon_n$ with each $\epsilon_k \in \{\pm 1\}$ such that $\sum_{k=0}^d \epsilon_k \equiv 2 \bmod{4}$.  Then the polynomial 
$$f(x):= \epsilon_0 + \epsilon_1 x + \dots + \epsilon_d x^d$$
is irreducible in $\mathbb{Z}[x]$. 
\end{lemma}

\begin{lemma}\cite[Lemma 4.2]{TiozzoGaloisConjugates} \label{l:Tiozzo4.2}
Let $f(x) = 1 + \sum_{k=1}^d \epsilon_k x^k$ be a polynomial with $\epsilon_k \in \{\pm 1\}$ for all $1 \leq k \leq d$ and $\epsilon_k = -1$ for some $k$.  If $f(x)$ is irreducible in $\mathbb{Z}[x]$, then for all $n \geq 1$, the polynomial $f(x^{2^n})$ is irreducible in $\mathbb{Z}[x]$.  
\end{lemma}

\section{Quadractic maps, iterated function systems, and renormalization}
\label{sec:connections}


In this section we elaborate on the connections to quadratic maps and
iterated function systems, including consequences of kneading theory for
the combinatorial itineraries we will study in this work. We close the
section with Proposition \ref{p:pers}, which is a sufficient criterion for
the Persistence Theorem \ref{mainthm:closurepersistence}.

\subsection{Iterated function system description of the Thurston set} \label{ss:IFSdescription}

Here we associate a limit set $\Lambda_z$ to a nonzero complex parameter
$z$ in the open unit disk $\mathbb D$. 
A point $z \in \mathbb{D} \setminus \{0\}$ defines a contracting iterated function system (IFS) generated by the two maps
$$f_0^z:x \mapsto zx+1, \quad f_1^z:x \mapsto zx-1.$$ 
The \emph{attractor} or \emph{limit set} $\Lambda_z$ of this IFS is defined
to be the unique fixed, nonempty, compact set $S \subset \mathbb{C}$ such
that $S = f_0^z(S) \cup g_1^z(S)$.  The existence and uniqueness of this
attractor and statements of some of its fundamental properties are due to Hutchinson \cite{Hutchinson}.  It is straightforward to see that:

\begin{lemma} \cite[Lemma 3.1.1]{ckw} \label{l:boundsonlimitset} 
  The limit set $\Lambda_z$ associated to 
  $z\in \mathbb D\setminus\{0\}$ is contained in the open ball of radius
  $\frac1{1-|z|}$ around the origin in the compex plane.
\end{lemma}

Our work is motivated by Tiozzo's description of $\Omega_2^{cp}\cap\mathbb{D}$ in \cite{TiozzoGaloisConjugates}, which is as follows:
\begin{theorem}\cite{TiozzoGaloisConjugates}\label{t:tiozzo}
  The Thurston set $\Omega_2^{cp}$ intersected with $\mathbb D$ is equal to
  the set of all compex numbers $z$ whose associated limit set $\Lambda_z$
  contains the origin. 
 \end{theorem}

Milnor and Thurston showed that any tent map $f_\beta$ is semiconjugate to
a quadratic map $g_c: x\mapsto x^2+c$ for $c\in [-2, 1/4]$, and this
semiconjugacy preserves the data of the Markov coding and hence the
entropy \cite{MilnorThurston}.
The {\em kneading series} of a quadratic map $g_c$ and a number $x$ is a power series
\[K(x, t) = 1 + \sum_{n=1}^{\infty}\eta_nt^n,\]

where $\eta_n(1)$ is the cumulative sign $\eta_n(x) = \prod_{i=0}^{n-1} sign(g_c^i(x))$.

The {\em kneading determinant} is
\[K_c(t) =  \begin{cases}
K(c,t) \textrm{ if the critical point is not periodic under }f_c, \\
\lim_{C \to c^+}K(C,t) \textrm{ if the critical point is periodic under }f_c\\
\end{cases}.\]

When $g_c$ has periodic critical orbit, $K_c(t)=\frac{P_{c, \text{knead}}(t)}{1-t^n}$, where $P_{c, \text{knead}}$ is called the {\em kneading polynomial}.

\begin{remark}\label{rem:knead-and-parry} The semiconjugacy between $f_\beta$ and $g_c$ sends the critical value $1$ to $c$, and intervals $I_o^\beta$ and $I_1^\beta$ to $[0, \infty)$ and $(-\infty, 0]$ respectively. Hence, $\eta_n(c)=s_w(n)$, which implies that:
\[(t-1)t^{n-1}P_{c, \text{knead}}(t^{-1})=P_\beta(t).\]
  \end{remark}

The following are classical from kneading theory; see \cite{MilnorThurston}:
\begin{proposition}\label{prop:quad}
  Fix a parameter $c$ in $[-2, 1/4]$ with associated quadratic map $g_c:
  x\mapsto x^2+c$ with entropy $h$, and let $s=e^h$ be the growth rate of
  this map. 
  \begin{enumerate}
  \item \cite[Theorem 13.1, Corollary 13.2]{MilnorThurston} 
    The growth rate 
    $s$ is a continuous function of $c$.
  \item \cite[Theorem 13.1, Corollary 13.2]{MilnorThurston} $s>\sqrt{2}$ if
    and only if $c<-1$.
  \item \cite[Theorem 6.3]{MilnorThurston} If $s>1$, $1/s$ is the smallest positive
    root of the kneading polynomial $P_{c, \text{knead}}$.
  \end{enumerate}\qed
\end{proposition}

Furthermore, 
Milnor and Thurston introduce an ordering on the additive group
$\mathbb{Z}[[t]]$ of formal power series with integer coefficients is
defined by setting $\alpha = a_0 + a_1t + \dots > 0$ whenever $a_0 = \dots
= a_{n-1} = 0$ but $a_n > 0$ for some $n \geq 0$ \cite{MilnorThurston}.
This induces an ordering on formal power series as follows: if $a,b$ are
distinct formal power series with integer coefficients, then 
$a>b$ if and only if $a-b>0$. 

The following is an immediate consequence of \cite[Section
13]{MilnorThurston}. We include proofs for completeness. 

\begin{lemma} \label{l:kneadingdetsgrowthrate}
For tent maps, the kneading determinant is a monotonically decreasing function of the growth rate.  
\end{lemma}

\begin{proof}
For the real one-parameter family of maps $f_a(x) = (x^2-a)/2$,
\cite[Theorem 13.1]{MilnorThurston} asserts that the kneading determinant
$D(f_a) \in \mathbb{Z}[[t]]$ is monotonically decreasing as a function of the
parameter $a$; and Corollary 13.2 asserts the growth rate is monotonically
increasing as a function of $a$.  The family of maps  $\{f_a\}$ takes on
all possible growth rates; this can be seen from the fact that  $f_a$ is
conjugate to the map $g_{(-a/4)}(z) = z^2+(-a/4)$ via the conjugation map
$h(z)=z/2$, growth rate is a continuous function of $c$ (Proposition~\ref{prop:quad}(1)), and the Intermediate Value Theorem. 
\end{proof}

\begin{lemma} \label{l:kneadingdetswordsrelationship}
Let $f_\beta$ be a tent map with kneading determinant $a$ and let $w$ be
the $\beta$-itinerary; let $f_{\beta'}$ be a tent map with kneading
determinant $b$ and let $w'$ be the
$\beta'$-itinerary. If $a > b$, then $w <_E w'$. 
\end{lemma}

\begin{proof}
  From Remark~\ref{rem:knead-and-parry}, $a= 1+\sum_{i=1}^{\infty} s_w(i)
  t^i$, $b=1+\sum_{i=1}^{\infty} s_{w'}(i) t^i$. Let $n$ be the smallest
  natural number such that $s_w(n) \neq s_{w'}(n)$.  
  We must have $s_w(1)=s_{w'}(1)$, so we may assume $n \geq 2$. 
  By definition, $a>b$ implies $s_w(k)=s_{w'}(k)$ for all $k=1,\ldots,n-1$,
  and 
  $s_w(k)=(-1)^{w_{k-1}}s_w(k-1)$, so we must have $w_j =w'_j$ for $1 \leq j
  \leq n-2$ and $w_{n-1} \neq w'_{n-1}$.
  Since $s_w(n)>s_{w'}(n)$, the two possibilities are:
$$s_{w}(n-1)=s_{w'}(n-1) = +1, \quad  w_{n-1} = 0, \quad w'_{n-1} = 1, \textrm{ or}$$ 
$$s_{w}(1,n-1)=s_{w'}(n-1) = -1,\quad w_{n-1} = 1, \quad w'_{n-1} = 0.$$
In both cases, we have $w <_E w'$.
\end{proof}

Combining Lemma~\ref{l:kneadingdetsgrowthrate} and Lemma~\ref{l:kneadingdetswordsrelationship}, we have:

\begin{corollary} \label{prop:word_monotonicity}
If $1<\beta<\beta'\leq 2$, the $\beta$-itinerary $w$ and 
$\beta'$-itinerary $w'$ satisfy the inequality $w<_Ew'$.
\end{corollary}

\subsection{Renormalization} \label{s:renormalization}\label{sec:renormalization}

We will develop the proof of the Persistence Theorem~\ref{mainthm:closurepersistence} using a
combinatorial approach of Tiozzo \cite{TiozzoGaloisConjugates}. Certain
tent maps $f_\beta$ admit itineraries with strong combinatorial properties.
Due to the renormalization phenomenon, if the slope $\beta$ is at most 
$\sqrt2$ then it is impossible for any associated itinerary to satisfy this
strong combinatorial property. {\em Renormalization} is how we and Tiozzo
compensate for this obstruction. 

One of the renormalization or ``tuning'' procedures on the Mandelbrot set (see e.g. \cite[\S~7.2]{TiozzoGaloisConjugates}) implies the following:

\begin{lemma}[Renormalization Lemma]\label{l:renorm}
  If $\beta\in (1, \sqrt{2}]$, then $f_\beta$ 
  has periodic critical orbit of length $2k$ 
  if and only if $f_{\beta^2}$ has periodic critical orbit of length
  $k$. 
\end{lemma}
This is a well-known fact and we give a short proof below for completeness.

\begin{proof}
Consider the intervals $J_1^\beta=\left[2-\beta, {2\over 1+\beta}\right]$,
$J_2^\beta=\left[{2\over 1+\beta}, 1\right]$. Then it is straightforward to
check that $f(J_1)\subset J_2$, $f(J_2)\subset J_1$, and $f^2\colon
J_2\rightarrow J_2$ is a unimodal piecewise linear map with constant slope
$\beta^2$ and critical value equal to $1$, and is clearly conjugate to the
tent map of slope $f_{\beta^2}$. Hence, $f_\beta^{  n}$, the
$n$th iterate of $f_\beta$, fixes
1 
if and only if $(f_\beta^2\big|_{J_2})^{  n/2}$, the $n/2$-iterate of
the restriction, also fixes 1.
\end{proof}


The Renormalization Lemma~\ref{l:renorm} motivates the following
definition:
\begin{definition}
  \label{def:renormalizable}
A tent map $f_\beta$ with growth rate 
$\beta$ in the interval $(1,2]$ is defined to be
{\em renormalizable} if $\beta\leq \sqrt2$, and is otherwise {\em
nonrenormalizable}. 
\end{definition}

With the Renormalization Lemma~\ref{l:renorm}, 
we reduce the Persistence Theorem~\ref{mainthm:closurepersistence} to the
following proposition, which is essentially persistence restricted to
growth rates of nonrenormalizable tent maps:

\begin{proposition}\label{p:pers}
  Let $\beta \in (1,2]$ be the growth rate of a critically periodic tent
  map and let $z\in \mathbb{D}$ be a root of the Parry polynomial of
  $\beta$. Then for any real number $y$ satisfying $2 >y>\max\{\beta,
    \sqrt{2}\}$ and any real number $\epsilon>0$, there exist a real number
    $\beta'$ within $\epsilon$ of $y$ such that
  \begin{enumerate}
  \item one of the Galois conjugates of $\beta'$ is within distance $\epsilon$
   of $z$, and 
  \item a Parry polynomial of $\beta'$ is of the form $(x-1)f(x)$, where
    $f(x^{2^n})$ is irreducible in $\mathbb{Z}[x]$ for all natural numbers $n$.
  \end{enumerate}
\end{proposition}

The proof of Proposition \ref{p:pers} will appear at the end of Section
\ref{sec:persistence}. 
We confirm here that this proposition is indeed sufficient to prove 
the Persistence Theorem~\ref{mainthm:closurepersistence}.

\begin{proof}[Proof of the Persistence Theorem~\ref{mainthm:closurepersistence} from Proposition~\ref{p:pers}]
  Suppose $(z, \beta)$ is a pair in the Master Teapot $ \Upsilon_2^{cp}$.
  We want to show that if $y\in [\beta, 2]$, then $(z,
  y)\in\Upsilon_2^{cp}$.
    Note that it suffices to consider
  $y>\beta$. 

  Let $\epsilon>0$. By definition of the Master Teapot, there exists a 
  $\beta_0$ that is the growth rate of some tent map
  $f_{\beta_0}$ with periodic critical orbit, $z_0$ a Galois conjugate of
  $\beta$, such that $|z_0-z|<\epsilon$ and $|\beta_0-\beta|<\epsilon$.
  Then we
  may choose $\epsilon$ small enough that $y>\beta_0$ as well. 
  
  If
  $y>\sqrt2$ then Proposition~\ref{p:pers} along with the triangle inequality
  directly implies existence of 
  $\beta'$ within $\epsilon$ of $y$, such that one of the Galois
  conjugates of $\beta'$ is within $2\epsilon$ distance from $z$, as
  desired. 

  It remains to consider $y<\sqrt2$. Fix an integer $n$ 
  such that $y\in \left(2^{1\over
  2^{n+1}}, 2^{1\over 2^n}\right]$.  
  By the Renormalization Lemma~\ref{l:renorm}, $\beta_0^{2^n}$ is the
  growth rate of some tent map with periodic critical orbit, is clearly
  less than $y^{2^n}$, and
  $z_0^{2^n}$ is a Galois conjugate of $\beta_0^{2^n}$ because the Galois group
  consists of field automorphisms.  
  Since $\sqrt2<y^{2^n}\leq 2$,  
  again by Proposition~\ref{p:pers}, 
  we can find some $\beta'$ within
  $\epsilon$ of $y^{2^n}$, such that one of the Galois conjugates $z'$ of
  $\beta'$ is within $2\epsilon$ of $z^{2^n}$. The second condition in
  Proposition~\ref{p:pers} implies that all the $2^n$-th roots of $z'$ are
  Galois conjugates of $(\beta')^{1\over 2^n}$. The conclusion follows.
\end{proof}

\section{Dominant words}\label{sec:dominant} \label{sec:auxiliary}

In this section we will review Tiozzo's definition of dominant words in \cite{TiozzoTopologicalEntropy} and the properties of dominant words proved in \cite{TiozzoTopologicalEntropy}.

In \cite{TiozzoTopologicalEntropy}, Tiozzo considers the tent map with slope 2 $f_2$, then any points $x$ in unit interval has a correspondence with a infinite sequence $w_x$ in $\{0, 1\}^{\mathbb{N}}$, which is its Markov coding, where the two subintervals are $I_0=[0, 1/2]$ and $I_1=[1/2, 0]$. It is obvious that $f_2$ is semiconjugate to the shift map $\sigma$ under this correspondence, and $x<y$ implies that $w_x<_Ew_y$ where $<_E$ is the twisted lexicographical order defined earlier.

For any finite word $w$ of length $n$, let $x$ be the point on the 
unit interval labeled by $w^\infty$, then the {\em Cylinder set} $C_x$ is
defined as the  subinterval which contains $x$ and $f_2^{n+1}$ sends it
homeomorphically to the unit interval $I$; in other words, $C_x$ is the
closure of points with a coding starting with $w\cdot w_1$. Here $w_1$ is the first letter of $w$.

\begin{definition}\cite[Definition 10.4]{TiozzoTopologicalEntropy} A finite
  word $w$ in the alphabet $\{0, 1\}$ is called
  {\em dominant} if and only if it satisfies the following two conditions:
  \begin{enumerate}
  \item $w$ has positive cumulative sign.
  \item For any $1\leq k\leq n-1$, $f_2^{k}(C_x)$ lies to the left of $C_x$ and their interiors are disjoint.
  \end{enumerate}
\end{definition}

Now, because $x<y$ implies $w_x<_Ew_y$, and $f_2^k(C_x)$ is the closure of points whose itineraries start with a suffix of $w$ followed by the first letter of $w$,
we can make the definition of dominance more explicit as below:

\begin{lemma} \label{l:dominantequivalentcharacterization}
  \label{lem:equivalentdefinitiondominance}
	Let $w$ be a word in the alphabet $\{0,1\}$ that starts with $10$
	and has positive cumulative sign.  Then $w$ is dominant if and only
	if for any proper suffix $b$ of $w$, the word $(b1)$ is (strictly)
	smaller than the prefix of $w$ of length $|b|+1$ in the twisted
	lexicographical ordering $<_E$. \qed
\end{lemma}

\begin{definition}
A word $w$ in the alphabet $\{0,1\}$ is \emph{irreducible} if there exists no shorter word $w_0$ in the alphabet $\{0,1\}$ and integer $n \geq 2$ such that $w=(w_0)^n$. 
\end{definition}

The main result we will cite from Tiozzo's work is the following, which can
be read from the proof of \cite[Theorem 10.5]{TiozzoTopologicalEntropy} in
\cite[Section 10.1]{TiozzoTopologicalEntropy}. One can read it from the first paragraph of the proof of Lemma 10.6 on page 689, and the third paragraph of the proof of Proposition 10.5 on page 692.

\begin{proposition}[{\cite[Section 11.2]{TiozzoTopologicalEntropy}}]
  \label{prop:densedominant}
  If $\beta\in(\sqrt{2},2]$ and $w$ is a word in the alphabet $\{0,1\}$
  such that
  $w^\infty$ is a $\beta$-itinerary, then for any positive integer 
  $n$ there exists a word $w'$ in the alphabet $\{0,1\}$ that is a power of
  some dominant word such that $w^nw'$ is also a dominant word. 
\end{proposition}

Tiozzo uses this observation to prove that the growth rates associated to
dominant words are dense in the interval $[\sqrt2,2]$.

The following lemma is straightforward to show from calculation:
\begin{lemma}\label{lem:firstletters}
Any $\beta$-itineraries for $\beta\in (1, 2]$ start with $10$.\qed
\end{lemma}

Because of Lemma~\ref{lem:firstletters},  the dominant itineraries obtained in Proposition~\ref{prop:densedominant} must satisfies the assumptions of Lemma~\ref{l:dominantequivalentcharacterization}.

\section{Persistence}\label{sec:persistence}

The goal of this section is to prove Proposition~\ref{p:pers}, which is a
version of persistence restricted to growth rates of nonrenormalizable tent
maps; that
is, growth rates 
that are larger than $\sqrt2$. As discussed in Section
\ref{sec:connections} following the statement of the proposition, the
Persistence Theorem~\ref{mainthm:closurepersistence} follows by
renormalization. 

\subsection{Constructing dominant extensions}

The development of persistence for growth rates of nonrenormalizable tent
maps 
hinges on a series of technical combinatorial lemmas. The goal of this
subsection is the proof of the Extension Lemma
\ref{lem:extension}.

\begin{proposition}
Assume $w_1$ is dominant, $w_2$ is admissible and irreducible, 
$n$ is a positive integer such that 
\[
	2n|w_2| > |w_1| > n|w_2|,
\]
 $w_1^\infty>_E w_2^\infty$, and $w_2^n$ has positive cumulative sign. 
Then $(w_1w_2^n)^{\infty}$ is admissible.
\label{prop:thekeylemma}
\end{proposition}

\begin{proof}
It suffices to show that \[\sigma^k(w_1w_2^n)^\infty\leq_E (w_1w_2^n)^\infty\] for all $k<|w_1|+n|w_2|$.
  If $1<k<|w_1|$, denote by $b$ the proper suffix of $w_1$ of length $|w_1|-k$. Then $(b1)$ is a
  prefix of $\sigma^k (w_1w_2^n)$
  because the first
  letter of $w_2$ is 1 by admissibility and Lemma \ref{lem:firstletters}.
  By dominance of $w_1$ and Lemma \ref{l:dominantequivalentcharacterization}, $(b1)$ is
  smaller than the prefix of $w_1$ of length $|b|+1$ in the twisted lexicographical ordering, 
  which proves
  \[\sigma^k(w_1w_2^n)=bw_2^n<_E w_1\] and provides the desired inequality. 

  If $k=|w_1|$, for contradiction, see that existence of $n$ such that $w_2^n\geq_E w_1$ implies
  \[
    w_2^\infty <_Ew_1^\infty \leq_E (w_2^n)^{\infty}=w_2^\infty,
  \]
  which is impossible given the assumption that $w_2^\infty$ is smaller than
  $w_1^\infty$ in the twisted lexicographical ordering. Thus, \[\sigma^{|w_1|}(w_1w_2^n)^\infty
  =w_2^n (w_1w_2^n)^\infty <_E
  (w_1w_2^n)^\infty.\] 

  Lastly, we consider the shift by $k$ where $|w_1|<k<|w_1|+n|w_2|$. Let $r=k-|w_1|$, so that
  $1<r<n|w_2|$. 
  See that $\sigma^rw_2^n>_Ew_1$ is impossible, because $\sigma^rw_2^n>_Ew_1$ 
  and admissibility of $w_2$ implies
  \[w_1^\infty<_E\sigma^r(w_2)^\infty\leq_E w_2^\infty,\] a contradiction. We conclude that $\sigma^rw_2^n\leq_E w_1$.
  If this inequality is strict, we are done: we would have 
  \[\sigma^k(w_1w_2^n)
  =\sigma^{|w_1|+r}(w_1w_2^n)=\sigma^rw_2^n<_Ew_1\] as desired.

  We must now consider when this inequality is not strict; in other words,
  $\sigma^rw_2^n$ is a prefix of $w_1$. We will need to prove that such
  a string must always have cumulative negative sign. If it does, then $|w_1|-r<|w_1|$
  implies 
  \[
    \sigma^{|w_1|-r}(w_1w_2^n)^\infty \leq_E (w_1w_2^n)^\infty 
  \]
  by dominance of $w_1$ discussed above. Then by Remark~\ref{rem:flippingsigns} and the fact that
  $\sigma^rw_2^n$ has negative cumulative sign,
  \begin{align*}
    (w_1w_2)^\infty =\sigma^rw_2^n\sigma^{|w_1|-r}w_1 w_2^n (w_1w_2^n)^\infty &\geq_E
    \sigma^rw_2^n(w_1w_2^n)^\infty \\
    &= \sigma^{k-|w_1|}w_2^n(w_1w_2^n)^\infty \\
    &= \sigma^k(w_1w_2^n)^\infty.
  \end{align*}

  It remains to prove that if $\sigma^rw_2^n$ is a prefix of $w_1$, then it cannot have cumulative
  positive sign. Consider the suffix $b=\sigma^rw_2^n$ of $w_2^n$. Since $w_2^n$ is admissible,
  $b\leq_E a$, where $a$ is the prefix of $w_2^n$ of the same length. 
  Since $w_2^\infty <_E
  w_1^\infty$, moreover $a$ is smaller than or equal to the prefix of $w_1$ of the same length,
  which is assumed to be equal to $b$. Then $b\leq_E a\leq_E b$ implies equality, and we conclude
  $w_2^n=ac=db=da$. 

  Now 
  \[
	  w_2^\infty=(ac)^\infty=(da)^\infty\geq_E a\concat (da)^\infty,
  \]
  implying
  \[
	  (ca)^\infty\geq_E(da)^\infty=w_2^\infty\geq_E (ca)^\infty
  \]
  because we assumed that $a$ has positive cumulative sign (see Remark \ref{rem:flippingsigns}) and $w_2$ is admissible, hence 
  $(ca)^\infty=(ac)^\infty$. Then
  \[
	  w_2^\infty= (ac)^\infty =a\concat(ca)^\infty=a\concat(ac)^\infty=a^2(ca)^\infty =
	  \cdots=a^\infty
  \]
  implies $a=w_2^m$ for some $m$ because $w_2$ is irreducible. 
  Then
  $w_1=af = w_2^mf$ for some suffix $f$, and again by dominance of $w_1$ and
  Lemma \ref{lem:equivalentdefinitiondominance},
  \[
	  w_1^\infty=(w_2^mf)^\infty=w_2^m(fw_2^m)^\infty\leq_E
	  (w_2^mw_1)^\infty=w_2^{2m}(fw_2^m)^\infty\leq_E\dots\leq_E w_2^\infty,
  \]
  which contradicts the assumption that $w_1^\infty>_Ew_2^\infty$. 
\end{proof}

Now we want to further make sure that $(w_1w_2^n)^\infty$ is a
$\beta$-itinerary and has a Parry polynomial with a large irreducible factor.
However, this would require some slight modification of the construction as
below:

\begin{definition}
  We say that a string 
  $v$ is an {\em extension} of a 
  word $w$ if $w$ is a proper prefix of
  $v$. If $v$ is finite then such a $v$ is a {\em finite extension} of $w$. 
\end{definition}

\begin{lemma}[Extension]
  \label{lem:extension}
Let $w_1$ be dominant such that $w_1>_E10\cdot 1^{|w_1|-2}$ and $w_2$ be 
admissible and irreducible, $w_1^\infty>_E w_2^\infty$, and assume
there exists an $m$ such that 
\[
	2m|w_2|>|w_1|>m|w_2|.
\]
Then there exists 
a finite extension $w_1'$ of $w_1$ and an integer $m'\geq m$ such that $(w_1'w_2^{m'})^\infty$ is
admissible, $|w_1'|>m'|w_2|$, and
$P(z^{2^k})/(z^{2^k}-1)$ is an irreducible polynomial for any $k\geq 0$, where $P$ is the Parry polynomial of $(w_1'
    w_2^{m'})$. 

\end{lemma}

The following Lemma will give us a recipe for extending $w_1$, as needed
for the Extension Lemma.

\begin{lemma}
  Let $w$ be a dominant word, such that $w_1>_E10\cdot 1^{|w_1|-2}$.
  Then the words
  \[
    w\concat 10\concat 1^\kappa\concat 10\concat 1^{|w|}\concat
    \mathbf{01}\concat 1^{|w|} \text{ and }w\concat 10\concat 1^\kappa
    \concat 10\concat 1^{|w|}\concat \mathbf{10}\concat 1^{|w|}
  \]
  for any odd natural number $\kappa>|w|$, and
  \[
    w\concat 1^\kappa\concat 10\concat 1^{|w|}\concat \mathbf{01}\concat
    1^{|w|} \text{ and }w\concat 1^\kappa \concat 10\concat 1^{|w|}\concat
    \mathbf{10}\concat 1^{|w|}
  \]
  for any even natural number $\kappa>|w|$, are all dominant extensions of $w$.

  Moreover, for each $\kappa$, the sums of the coefficients of the kneading
  polynomials for the two extensions differ by 2.
  \label{lem:extension_options}
\end{lemma}


\begin{proof}
  The parity condition on $\kappa$ is to guarantee that the new word has an
  even number of 1s, which is part of the definition of dominance.

  We apply the alternate definition of dominance from Lemma
  \ref{lem:equivalentdefinitiondominance}. Let $w'$ be one of the possible
  extensions in the statement of the Lemma. Let $b$ be any suffix of $w'$.
  If a prefix of $b$ is a suffix of $w$, then $(b1)$ is smaller than the
  prefix of $w'$ of the same length in the twisted lexicographical ordering
  by dominance of $w$ and the construction of $w'$. If not, then if $b$
  starts with 0 or $11$, and the desired inequality is immediate, so the
  interesting case is if $b$ starts with $10$ and no prefix of $b$ is a
  suffix of $w$. By construction, including our choice of $\kappa>|w|$ in
  the $\kappa$ odd case, we are comparing $10\cdot 1^{|w|-1}$ with $w\cdot
  1$, and the former must be smaller by assumption.

  For any natural number $\kappa$, odd or even, there are now two choices
  to extend $w$ to a dominant word. The two choices only differ by an
  exchange of $01$ with $10$ in one position.  This exchange will change
  the sum of the coefficients of the kneading polynomials by a factor of 2.
\end{proof}


\subsubsection{Proof of extension lemma}

\begin{proof}[Proof the Extension Lemma \ref{lem:extension}]
  We need to choose for $w_1'$ one of the extensions of $w_1$ from Lemma
  \ref{lem:extension_options}, and select $n$, $\kappa$, and
  $m'$ so that $|w_1'|$ has length $2^n-1-m'|w_2|$ and
  \[
    2m'|w_2|>|w_1'|>m'|w_2|.
  \]
  To do so, first define constants $C_1=1+|w_1|+m|w_2|$ and $C_2=|w_2|$. 
  Then choose $n$ for which 
  \[2^n>\max\{C_2(10m+3)+C_1,18C_2+C_1\}\] and define 
  \begin{equation}
    \label{eqn:def_k_n}
    k_n=\left\lceil \frac{2^n-C_1}{2C_2}\right\rceil-2,\qquad 
    k_n'=\left\lceil \frac{2^n-C_1}{2C_2}\right\rceil-3. 
  \end{equation}
  The two options $k_n$ and $k_n'$ are needed for parity reasons. 
  Choosing $2^n>C_2(10m+3)+C_1$ ensures that 
  \begin{equation} 
    k_n>k_n'>10m, 
    \label{eqn:bound_k_n_below}
  \end{equation}
  which becomes useful later in the proof when we define the length of the
  extension. The choice of $2^n>18C_2+C_1$ and the definition of $k_n,k_n'$
  ensures (respectively) that 
  \begin{equation}
    3k_n>3k_n'>\frac{2^n-C_1}{C_2}>2k_n>2k_n'.
    \label{eqn:controlling_extension_length}
  \end{equation}
  Let $m'=k_n+m$ if this is even, and else, replace $k_n$ with $k_n'$. We
  will proceed with the notational choice $m'=k_n+m$ and assume $m'$ is
  even, but note that the needed inequalities hold for both $k_n$ and
  $k_n'$. 

  Now, replacing $C_1,C_2$ with their definitions, applying Equation
  \eqref{eqn:controlling_extension_length}, and invoking the assumed
  relationship between $|w_1|$ and $|w_2|$, we see that  
  \begin{equation*}
    3m'|w_2|>3k_n|w_2|+m|w_2|+|w_1|>2^n-1>2k_n|w_2|+m|w_2|+|w_1|>2m'|w_2|,
  \end{equation*}
  which implies 
  \begin{equation}
    2m'|w_2|>2^n-1-m'|w_2|>m'|w_2|.
    \label{eqn:needed_to_apply_prop}
  \end{equation}
  We now adjust the extension $w_1'$ of $w_1$ to have length
  $|w_1'|=2^n-1-m'|w_2|$, so that $(w_1'w_2^{m'})$ has total length
  $2^n-1$. 

  If $|w_1|$ is odd, then 
  $\kappa=(2^n-1-m'|w_2|)-6-3|w_1|$ is even, as needed for 
  \[
    w\concat 1^\kappa\concat 10\concat 1^{|w|}\concat {\mathbf{01}}\concat
    1^{|w|}\quad \text{ and }\quad w\concat 1^\kappa \concat 10\concat 1^{|w|}\concat
    {\mathbf{10}}\concat 1^{|w|}
  \]
  to both be dominant extensions of $w_1$ by Lemma
  \ref{lem:extension_options}, each of length $2^n-1-m'|w_2|$.

  If $|w_1|$ is even, then 
  $\kappa=(2^n-1-m'|w_2|)-4-3|w_1|$ is odd, as needed for 
  \[
    w_1\concat 1^\kappa\concat 10\concat 1^{|w_1|}\concat {\mathbf{01}}\concat
    1^{|w_1|} \quad\text{ and }\quad w_1\concat 1^\kappa \concat 10\concat 1^{|w_1|}\concat
    {\mathbf{10}}\concat 1^{|w_1|}
  \]
  to both be dominant extensions of $w_1$ by Lemma \ref{lem:extension_options}, each of length
  $2^n-1-m'|w_2|$. In all the above cases, $\kappa>|w_1|$ follows from Equation
  \eqref{eqn:bound_k_n_below}. 

  For each choice, $w_1^\infty>_Ew_2^\infty$ implies
  $w_1'^\infty>_Ew_2^\infty$, and $w_2^{m'}$ has positive cumulative sign
  because we ensured that $m'$ is even. Combined with Equation
  (\ref{eqn:needed_to_apply_prop}), we have all the necessary hypotheses to
  apply Proposition \ref{prop:thekeylemma} and conclude that
  $(w_1'w_2^{m'})^\infty$ is admissible. We also designed $w_1'$ so that
  $|w_1'|>m'|w_2|$.

  The sum of the coefficients of the kneading polynomial of $w_1'w_2^{m'}$
  is even, because it has $2^n$ coefficients, each of which is either $-1$
  or $+1$.  By the final observation in Lemma \ref{lem:extension_options},
  we can choose the extension so that the sum of the 
  coefficients of the kneading polynomial for $w_1'w_2^{m'}$ is equivalent to
  $2 \mod 4$.  Since the kneading polynomial has degree $2^n-1$, we apply
  Lemma \ref{l:Tiozzo4.2} to conclude irreducibility. 
\end{proof}

\subsection{Controlling Galois conjugates and entropies of concatenations}

\begin{lemma}[Nearby roots] \label{l:smallrootsdontmovemuch}
Let $w_2$ be a 
word whose Parry polynomial has a root at $z_0 \in \mathbb{D}$.  Then for any $\epsilon > 0$, there exists an integer $N = N(\epsilon,w_2) \in \mathbb{N}$ such that $n > N$ implies that for every 
word $w_1$ 
for which $w_1w_2^n$ is admissible,
the Parry polynomial associated to $(w_1w_2^n)$ has a root within distance $\epsilon$ of $z_0$. 
\end{lemma}

\begin{proof}
First, for any word $w$, denote the Parry polynomial associated to $w$ by $P_w$.  
Let $D$ be the closed disk of radius $\epsilon$ centered at $z_0$, and let $C$ be the boundary of $D$.  Without loss of generality, assume that $\epsilon$ is small enough that $D \subset \mathbb{D}$, and that $D$ contains no root of $P_{w_2}$ except $z_0$.  

For any $n \in \mathbb{N}$, it is straightforward to see that 
\[
	P_{w_1w_2^n}(z) = z^{n|w_2|}P_{w_1}(z) + \left(z^{(n-1)|w_2|} + z^{(n-2)|w_2|}+ \dots + 1 \right)P_{w_2}(z).
\]

\noindent Set $\alpha = \min_{z \in C} |P_{w_2}(z)|$, which exists and is positive by compactness
and the assumption that $D$ contains no root of $P_{w_2}$ except $z_0$. 
Set 
\[ 0<\beta  
 := \min_{z \in C} \left(1-|z|^{|w_2|}\right)/\left(1+|z|^{|w_2|}\right).\]
Then for all $z \in C$, we have 
\[
	\left|\left(z^{(n-1)|w_2|} + z^{(n-2)|w_2|}+ \dots + 1\right)P_{w_2}(z)\right| \geq \left|
	\frac{1-(z^{|w_2|})^n}{1-z^{|w_2|}} \right| \alpha \geq
	\frac{1-|z^{|w_2|}|}{1+|z^{|w_2|}|}\alpha \geq 
	\beta \alpha>0,
\]
	where the middle nonstrict inequality follows the triangle inequality and that
	$\left|z^{|w_2|}\right|<1$. 

Set $1> m: = \max_{z \in D} |z|$.  Also for all $z \in C$, since all coefficients of $P_{w_1}$ have absolute value at most $3$,
\[
	\left| z^{n|w_2|}P_{w_1}(z) \right|  \leq | z|^{n|w_2|} \left(1+3 \sum_{i=0}^{\infty} |z|^i \right) \leq
	m^{n|w_2|} \left(1+3 \sum_{i=0}^{\infty} m^i \right).
\]

\noindent Therefore, for sufficiently large $n \in \mathbb{N}$ depending only on $w_2$, we have 
\[
	\left| z^{(n-1) |w_2|}P_{w_1}(z) \right|  < \frac{\beta \alpha}{2}.
\]
Consequently, the winding number around $0$ of the image of $C$ under $P_{w_1w_2^n}$ equals the
winding number around $0$ of the image of $C$ under the map
$$z \mapsto \left(z^{(n-1)|w_2|} + z^{(n-2)|w_2|}+ \dots + 1\right)P_{w_2}(z).$$  
The winding number of the image around $0$ is related to number of zeros via the Argument Principle; for a holomorphic function $f$ and a simple closed contour $\Gamma$, the number $N$ of zeros of $f$ inside $\Gamma$ is given by
\begin{equation} \label{eq:argumentprinciple}
N = \frac{1}{2\pi i} \int_{\Gamma} \frac{f^{\prime}(z)}{f(z)}dz = \frac{1}{2\pi i} \int_{f(\Gamma)} \frac{dw}{w},
\end{equation} where $w=f(z)$.  
Since $P_{w_2}$ has a root in $D$, this implies $P_{w_1w_2^n}$ also has a root in $D$ for
sufficiently large $n$.  
 \end{proof}

\subsection{Dominant approximations of growth rates}

Finally, we prove that we can approximate $y\in(\sqrt2,2)$ with growth rates
corresponding to extensions of dominant strings. 

\begin{lemma}
  Let $w_1$ be an admissible word such that $w_1^\infty$ is a
  $\beta$-itinerary for some $\beta>1$.  For any $\epsilon>0$, there
  exists an integer $N = N(\epsilon,w_1)$ such $n>N$ implies that for
  every word $w_2$ for which $(w_1^n w_2)^\infty$ is a
  $\beta'$-itinerary, $\beta'$ is within distance $\epsilon$ of $\beta$.
  \label{l:leadingrootdontmovemuch}
\end{lemma}
\begin{proof}
  Let $w'$ be the $(\beta-\epsilon)$-itinerary, and let $w''$ be the
  $(\beta+\epsilon)$-itinerary. Let $N$ be large enough that the
  $N|w_1|$-prefix of $w'$, $w_1^\infty$ and $w''$ are all distinct. Then we
  must have $w'<_E(w_1^nw_2)^\infty<_Ew''$ so $\beta'\in (\beta-\epsilon,
  \beta+\epsilon)$. 
\end{proof}

Combining Lemma~\ref{l:leadingrootdontmovemuch},
Proposition~\ref{prop:densedominant}, and the fact that slopes of tent maps
with periodic critical orbits are dense, the following result is evident:
 
\begin{lemma}[Dominant approximations]
  For all $y\in(\sqrt2,2)$ and all $\epsilon>0$, there exists a sequence of
  dominant words $(w_n)_{i=1}^\infty$ such that for any admissible
  extension $w_n'$ of $w_n$, including the empty extension, if
  $(w_n')^\infty$ is a $\beta$-itinerary, then $\beta$ is within
  $\epsilon$ of $y$. 
  \label{lem:dominant_approximation}
\end{lemma}

\subsection{Proof of Proposition~\ref{p:pers}}

Now we prove Proposition~\ref{p:pers}, which will finish the proof of the Persistence Theorem~\ref{mainthm:closurepersistence}.

\begin{proof}[Proof of Proposition~\ref{p:pers}]
  Let $w$ be an irreducible word in the alphabet $\{0,1\}$ such that 
  $w^\infty$ is the $\beta$-itinerary. 
  Then $z$ is a root of the
  Parry polynomial of $w$, since this is equal to the Parry polynomial of
  $\beta$.  
  If $y = \beta$ the statement is trivial, so
  assume $y > \beta$.  Fix $$0<\epsilon<\frac{y-\beta}2.$$ Construct the
  sequence of dominant words $(w_n)$ as in 
  the Dominant Approximations Lemma~\ref{lem:dominant_approximation}; 
  the words $w_n$ satisfy that for any admissible extension $w_n'$ of
  $w_n$, if $(w_n')^\infty$ is some $\beta'$-itinerary, then
  $\beta'$ is within $\epsilon$ of $y$. 
  We will show there is a subsequence
  of $(w_n)$ with corresponding extensions $(w_n')$ which meet this
  criteria and whose corresponding growth rates have controlled Galois
  conjugates.

  Passing to subsequences as needed, we may assume that $|w_n| \to \infty$
  as $n \to \infty$, since there are only finitely many words of bounded
  length. 
	
  For each $n$, let $M_n=\left\lceil \frac{|w_n|}{|w|}\right\rceil - 2 $.
  Then 
  \begin{equation} \label{eq:iffstatement}
    2M_n|w|\geq 2\left(\frac{|w_n|}{|w|}-2\right)|w| = 2|w_n|-4|w|. 
  \end{equation}
  Since $2|w_n|-4|w| > |w_n|$ if and only if $|w_n|>4|w|$, we have from
  equation (\ref{eq:iffstatement}) that 
  \begin{equation} \label{eq:foriff}
    |w_n|>4|w| \ \Longrightarrow \
    2M_n |w| > |w_n|.  
  \end{equation}
  Observe that	
  \[
    |w_n|
    =\frac{|w_n|}{|w|}|w| > \left( \left\lceil
    \frac{|w_n|}{|w|}\right\rceil-2 \right)|w| 
    =M_n|w| \quad \text{ for all }n
  \]
  and $|w_n| \to \infty$.  Therefore, for all $n$ large enough that
  $|w_n|>4|w|$, there exists a positive integer $M_n$ such that 
  \[
    2M_n|w| > |w_n| > M_n |w|.
  \]
  Note also that $M_n \to \infty$ as $n \to \infty$.  	
  Thus, for sufficiently large $n$, the hypotheses of 
  the Extension Lemma~\ref{lem:extension} 
  hold, using $w_n$ in place of $w_1$ and $w$ in place of $w_2$.  
  Then 
  we conclude there exists an integer $m_n'>M_n$ and a dominant extension
  $w_n'$ of $w_n$ so that $(w_n'w^{m_n'})^\infty$ is admissible and the
  polynomial 
  \[
    \frac{P_{w_n'w^{m_n'}}(x^{2^k})}{1-x^{2^k}}
  \]
  is irreducible for all $k\geq 0$, where $P_{w_n'w^{m_n'}}$ is the Parry
  polynomial of the admissible word $w_n'w^{m_n'}$. Hence, by
  Proposition \ref{p:min}, $(w_n'w^{m_n'})^\infty$ is a
  $\beta'$-itinerary, and we have the criteria needed to apply the final
  conclusion of 
  the Dominant Approximations Lemma~\ref{lem:dominant_approximation}); 
  that is, $\beta'\in [y-\epsilon,
  y+\epsilon]$.  Since $M_n \to \infty$ as $n \to \infty$ and
  $m_n^{\prime} > M_n$, we have $m_n' \to \infty$ as $n\to \infty$.  
  
  We know $z$ is a root of $P_w$ so
  by the Nearby Roots Lemma~\ref{l:smallrootsdontmovemuch}),
  for sufficiently large $n \in \mathbb{N}$, we conclude $P_{w_n'
  w^{M_n'}}$ has a root within $\epsilon$ of $z$. 
\end{proof}



\section{The unit cylinder and connectivity}\label{sec:cylinder}


\begin{proposition} \label{p:bottomlevelunitcircle}
The bottom level of the Master Teapot is the unit circle, i.e. $$ \Upsilon_2^{cp} \cap (\mathbb{C} \times \{1\}) = S^1 \times \{1\}.$$  
\end{proposition}

\begin{proof}
  We will first show $S^1 \times \{1\} \subset \Upsilon_2^{cp}$.  By
  Proposition \ref{p:pers}, there exists some growth rate $\beta \in (1,2)$
  of a critically periodic tent map such that $\beta$ has a Galois
  conjugate $z \in \mathbb{D}$ satisfying the condition that, for all $k
  \in \mathbb{N}$, every $2^k$-th root of $z$ is a Galois conjugate of
  $\beta^{1\over 2^k}$.   Repeatedly applying renormalization to $(z,
  \beta)\in \Upsilon_2^{cp}$ - by which we mean considering the set of
  $2^k$ points of the form $(z^{1/2^k},\beta^{1/2^k})$, all of which are in
  $\Upsilon_2^{cp}$, as $k \to \infty$ - and then taking the set of limit
  points, we get that $S^1 \times \{1\} \subset \Upsilon_2^{cp}$.

To show  $\Upsilon_2^{cp} \cap (\mathbb{C} \times \{1\}) \subset S^1 \times \{1\}$, suppose there exists a point $(y,1)\in\teapot$ such that $|y|\neq 1$. 
Since $1$ has no nontrivial Galois conjugates, $(y,1) \in \mathbb{C} \times \mathbb{R}$ must be the the limit of a sequence of points $(\alpha_n,\beta_n) \in \mathbb{C} \times \mathbb{R}$ such that $\beta_n$ is the growth rate of a superattracting tent map and $\alpha_n$ is a Galois conjugate of $\beta$.  
Thus, reindexing the sequence as necessary, we have that for any $k>0$,
there exists $\beta_{k}$ with $1<\beta_{k}<1+\frac{1}{k}$ with Galois
conjugate $\alpha_{k}$, so that $|\alpha_k-y|<\epsilon$. Now by
renormalization, $\beta_k^{2^{n_k}} \leq 2$ is the slope of a critically
periodic tent map, where $n_k$ is the maximal value of $N$ for which
$\beta_k^{2^{N}}\leq 2$.  The fact that $\alpha_k^{2^{n_k}}$ is a Galois conjugate of $\beta_k^{2^{n_k}}$ follows immediately from the definition of a Galois automorphism.  Thus $(\alpha_k^{2^{n_k}},\beta_k^{2^{n_k}}) \subset \Upsilon_2^{cp}$.  

Now, $|\alpha_k|$ is bounded away from $1$ for $k$ sufficiently large (because $\alpha_k \to y$), and $n_k\to\infty$ as $k \to \infty$, since $\beta_k \to 1$ as $k \to \infty$.  Consequently, either $\alpha_k^{2^{n_k}} \to 0$  or $\alpha_k^{2^{n_k}} \to \infty$ as $k\to \infty$. This is a contradiction because
$$\Omega_2^{cp} \subset \{z \in \mathbb{C}: 1/2 \leq z \leq 2\}$$
 by  \cite[Lemma 2.4]{TiozzoGaloisConjugates} 
and that the projection of $\Upsilon_2^{cp}$ onto the first coordinate is $\Omega_2^{cp}$.   
\end{proof}

\begin{proof}[Proof of the Unit Cylinder Theorem \ref{mainthm:unitcylinder}]

As in the proof of Proposition~\ref{p:bottomlevelunitcircle}, let $\beta$ be the slope of a critically periodic tent map such that $\beta$ has a Galois conjugate  $z \in \mathbb{D}$ and for all $k \in \mathbb{N}$, and every $2^k$-th root of $z$ is a Galois conjugate of $\beta^{1\over 2^k}$. Then we have
$$\left \{(z', \beta') \in \mathbb{C} \times \mathbb{R}^+ :  (z ')^{2^k}=z \textrm{ and }\beta'=\beta^{1\over 2^k} \textrm{ for some } k \in \mathbb{N}\right\}\subset \Upsilon_2^{cp}.$$ The Persistence Theorem~\ref{mainthm:closurepersistence} then implies that 
$$
\left \{(z', y) \in \mathbb{C} \times \mathbb{R}^+ : (z')^{2^k}=z, y\geq \beta^{1\over 2^k}\text{ for some } k \in \mathbb{N} 
\right \}\subset \Upsilon_2^{cp},$$ 
Taking the closure of the above set, we obtain that the unit cylinder
$\mathbb{S}^1 \times [1,2]$ is a subset of $\Upsilon_2^{cp}$.
\end{proof}

 \begin{proof}[Proof of the Connectedness Theorem \ref{mainthm:connected}]
Connectivity of the part of the Master Teapot outside of the unit cylinder
is due to Tiozzo \cite{TiozzoGaloisConjugates}.  More specifically, by
\cite[Lemma 7.3]{TiozzoGaloisConjugates}, for any point $(z,\beta) \in
\mathbb{C} \times \mathbb{R}$ such that $\beta$ is the growth rate of a
critically periodic tent map, $z$ is a Galois of $\beta$, and $|z| > 1$,
there exists a continuous path $(\gamma(x),x)$ in $\Upsilon_2^{cp}$
connecting $(z,\beta)$ to a point $(w,1)$ with $w \in S^1$.   Consequently,
since the unit cylinder is a subset of $\Upsilon_2^{cp}$ by the Unit
Cylinder Theorem \ref{mainthm:unitcylinder}, and since $\Upsilon_2^{cp}$ is
closed, this implies  $\Upsilon_2^{cp} \cap (\{z:|z| \geq 1\} \times
\mathbb{R})$ is connected.  By \cite{TiozzoGaloisConjugates}, 
the Thurston Set is connected and contains an open annulus containing $S^1$.  By the Persistence Theorem~\ref{mainthm:closurepersistence}, the projection to $\mathbb{C}$ of part of the top level of the Master Teapot that is inside the unit cylinder agrees with the Thurston Set, i.e. $\Upsilon_2^{cp} \cap (\mathbb{D} \times \{2\}) = (\Omega_2  \cap \mathbb{D}) \times \{2\}$.  Also by the Persistence Theorem~\ref{mainthm:closurepersistence}, the part of the Master Teapot inside the unit cylinder is connected.  Thus, the entire Master Teapot, $\Upsilon_2^{cp}$, is connected.  
 \end{proof}

\section{Gaps in the Thurston set} \label{sec:gaps}

Plots of finite approximations of the Thurston set consisting of the roots of all defining polynomials associated to superattracting tent maps of critical orbit length at most $n$, for fixed $n \in \mathbb{N}$, have  ``gaps" at certain algebraic integers, some of which are on the unit circle.  The Thurston set contains a neighborhood of the unit circle \cite{TiozzoGaloisConjugates}, but these gaps get filled in more slowly with $n$ than some other regions.  
See Figure \ref{fig:thurston_set} for a picture of a finite approximation of the Thurston set, and Figure \ref{fig:gaps} for a closeup of one such gap. In this section, we prove an arithmetic justification for gaps:

\begin{figure}[!h] 
\begin{center}
\includegraphics[width=0.7\linewidth]{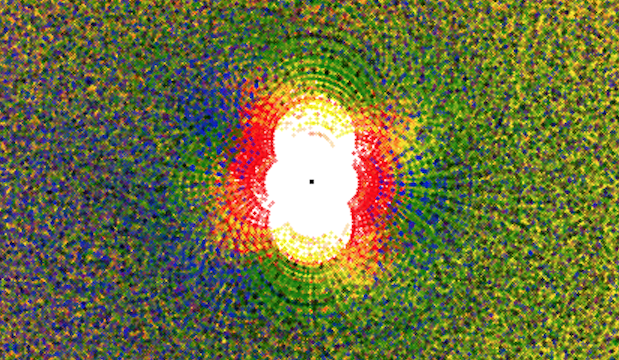}
\caption{A closeup of the how the ``gap'' around the point $i$ fills in as postcritical length increases, for an approximation of the Thurston set. The points are color-coded by the length of the associated post-critical orbit. Blue is the shortest, followed by green, yellow, orange, and finally red with the longest orbit, of length 23.
}
\label{fig:gaps1}
\end{center}
\end{figure}

\begin{theorem} \label{t:gaps}
Let $\alpha$ be an algebraic integer such that $\mathbb{Z}[\alpha]$ is a discrete subgroup of $\mathbb{C}$ and let $x \in \mathbb{Z}[\alpha]$. Set 
$c = \min \{|z| : z \in \mathbb{Z}[\alpha], z \neq 0\}$.
Suppose there exists a superattracting tent map with postcritical length $n$ whose growth rate has a Galois conjugate of the form $x+\epsilon$ for some $\epsilon \in \mathbb{C}$ with $|\epsilon| \leq \frac{1}{n+1}$.  
Then 

\begin{enumerate}
\item if $|x| \geq 1$, then 
$\displaystyle\frac{c}{(2n^2 + 3n+1) |x|^n e} \leq \epsilon.$
\item if $|x| \leq 1$, then 
$\displaystyle\frac{c}{(2n^2+3n+1) e} \leq \epsilon.$
\end{enumerate}
\end{theorem}

\begin{proof}
Fix $x \in \mathbb{Z}[\alpha]$ and suppose there exists a real number $\beta$ associated to a generalized PCF $\beta$-map with $m$ intervals and postcritical length $n$ that has a Galois conjugate of the form $x+\epsilon$ for some $\epsilon \in \mathbb{C}$ with $|\epsilon| \leq 1$.   

Then $\beta$ is the root of the associated Parry polynomial $P_{\beta,E}$;
$$0 = z^{n+1}-(a_0z^n + a_1 z^{n-1}+\cdots + a_n) - 1,$$ 
where $a_i \in \{-2,0,2\}$. Hence $(x+\epsilon)$ is also a root of $P_{\beta,E}$:
\begin{equation*} \label{eq:epsilonpoly2} 
0 = (x+\epsilon)^{n+1} - (a_0(x+\epsilon)^n + a_1(x+\epsilon)^{n-1}+ \cdots + a_n)-1.
\end{equation*} 
Therefore
\begin{align*}
	1-x^{n+1}+a_0x^n + \dots +a_n &= (x+\epsilon)^{n+1}-x^{n+1} -\big( a_0((x+\epsilon)^{n}-x^n) \\
	&\qquad + a_1((x+\epsilon)^{n-1}-x^{n-1}) +\dots + a_{n-1}((x+\epsilon)-x)\big).
\end{align*}
We have $1-x^{n+1}+a_0x^n + \dots +a_n \in \mathbb{Z}[\alpha]$, so $c \leq |1-x^{n+1}+a_0x^n + \dots +a_n|$.  Then by the triangle inequality, 

\begin{align}\label{eq:firsttriangleinequality}
\begin{split}
c &\leq |1-x^{n+1}+a_0x^n + \dots +a_n|  \\
& \leq |(x+\epsilon)^{n+1}-x^{n+1}| + |a_0||(x+\epsilon)^{n}-x^n| + |a_1||(x+\epsilon)^{n-1}-x^{n-1}| \dots |a_{n-1}||(x+\epsilon)-x|.
\end{split}
\end{align}

We now restrict to the case $|x| \geq 1$.  
For any $k \leq n+1$, by the binomial theorem, the triangle inequality, and $|\epsilon| \leq \frac{1}{n+1}$, 

\begin{align} \label{eq:biginequalitybigx}
\begin{split}
| (x+\epsilon)^k - x^k | &= \quad\left| \sum_{i=1}^k \begin{pmatrix} k \\ i \end{pmatrix} x^{k-i} \epsilon^i \right| \quad\leq\quad  \sum_{i=1}^k \left| \begin{pmatrix} k \\ i \end{pmatrix} x^{k-i} \epsilon^i \right| \\
&\leq \quad \sum_{i=1}^k \left| \frac{k^i}{(k-i)!} x^{k-i} \frac{1}{(n+1)^{i-1}} \epsilon   \right| \quad=\quad \sum_{i=1}^k \left| \left( \frac{k}{n+1} \right)^{i-1} \frac{k}{(k-i)!} \ \epsilon \ x^{k-i}\right| \\
&\leq\quad \epsilon k |x|^{k-1} \sum_{i=1}^k \frac{1}{(k-i)!}  \quad=\quad  \epsilon k |x|^{k-1}  \sum_{i=0}^{k-1} \frac{1}{i!} \\
&\leq \quad \epsilon k |x|^{k-1}  \sum_{i=0}^{\infty} \frac{1}{i!} \quad=\quad \epsilon k |x|^{k-1} e.
\end{split}
\end{align}

\noindent Combining equations (\ref{eq:firsttriangleinequality}) and (\ref{eq:biginequalitybigx}) yields

\begin{align*} 
\begin{split}
c & \leq  \epsilon (n+1) e |x|^n + 
|a_0| \epsilon n e |x|^{n-1}
+ \dots + 
|a_{n-1}| \epsilon 1 |x|^0 e| \\
& \leq  \epsilon (n+1) e |x|^n \left(1 + |a_0| + \dots + |a_{n-1}| \right) \\
& \leq  \epsilon (n+1) e |x|^n \left(1 +2n \right). \\
\end{split} 
\end{align*}

\noindent Thus for $|x|\ge 1$,
$$\frac{c}{e (1+2n) (n+1) |x|^n} \leq \epsilon.$$

We now restrict to the case $|x| \leq 1$.  In this case, the estimate (\ref{eq:biginequalitybigx}) becomes 
\begin{equation} \label{eq:littlexcase}
  \left|(x+\epsilon)^k - x^k\right| \leq 
  \epsilon ke.
\end{equation}
Combining equations (\ref{eq:firsttriangleinequality}) and (\ref{eq:littlexcase}) yields
$$c \leq \epsilon (n+1) e(1+|a_0| + |a_1| + \dots +|a_{n-1}|) \leq  \epsilon (n+1) e(1+2n) .$$ 
Hence, for $|x| \le 1$, 
$$ \frac{c}{(n+1)(1+2n) e} \leq \epsilon.$$
\end{proof}

\begin{proof}[Proof of the Gap Theorem \ref{mainthm:gaps}]
In view of Theorem \ref{t:gaps}, it suffices to classify the discrete subgroups of $\mathbb{C}$.
The classification of discrete subrings of $\mathbb{C}$ is well-known, and
we include it for completeness: firstly, because it is a discrete additive
subgroup, it is either $\mathbb{Z}$ or a lattice of rank $2$. If it is the
latter case, let $\{1, a\}$ be a basis of the lattice, then $a$ must be an
algebraic integer of degree $2$, in other words, the discrete subring must be of the form $\mathbb{Z}[a]$ where $a$ is an
algebraic integer of degree 2, hence it must be contained in the ring of integers of an algebraic field of degree 2. There are only 4 such rings of integers that contains some element not on the real line and has absolute value less than 2, which are
 $\mathbb{Z}[\sqrt{-1}]$, $\mathbb{Z}[\sqrt{-2}]$, $\mathbb{Z}[\sqrt{-5}]$, or
$\mathbb{Z}[\frac{1+\sqrt{-3}}2]$. 
\end{proof}

\section{$\Omega_2^{cp}$ and $\Omega_2^{pcf}$ are not equal}\label{sec:preperiodic}

In this section we prove the Two Thurston Sets Theorem~\ref{mainthm:prepernotequal}, that $\Omega_2^{cp}$ and $\Omega_2^{pcf}$ are not equal. A finite approximation of $\Omega_2^{cp}$ is shown in Figure \ref{fig:thurston_set}, and a finite approximation of $\Omega_2^{pcf}$ is shown in Figure \ref{fig:preperiodic}.

\begin{figure}[!hb]
\begin{center}
\includegraphics[width=\linewidth,trim={15cm 19cm 10cm 19cm},clip]{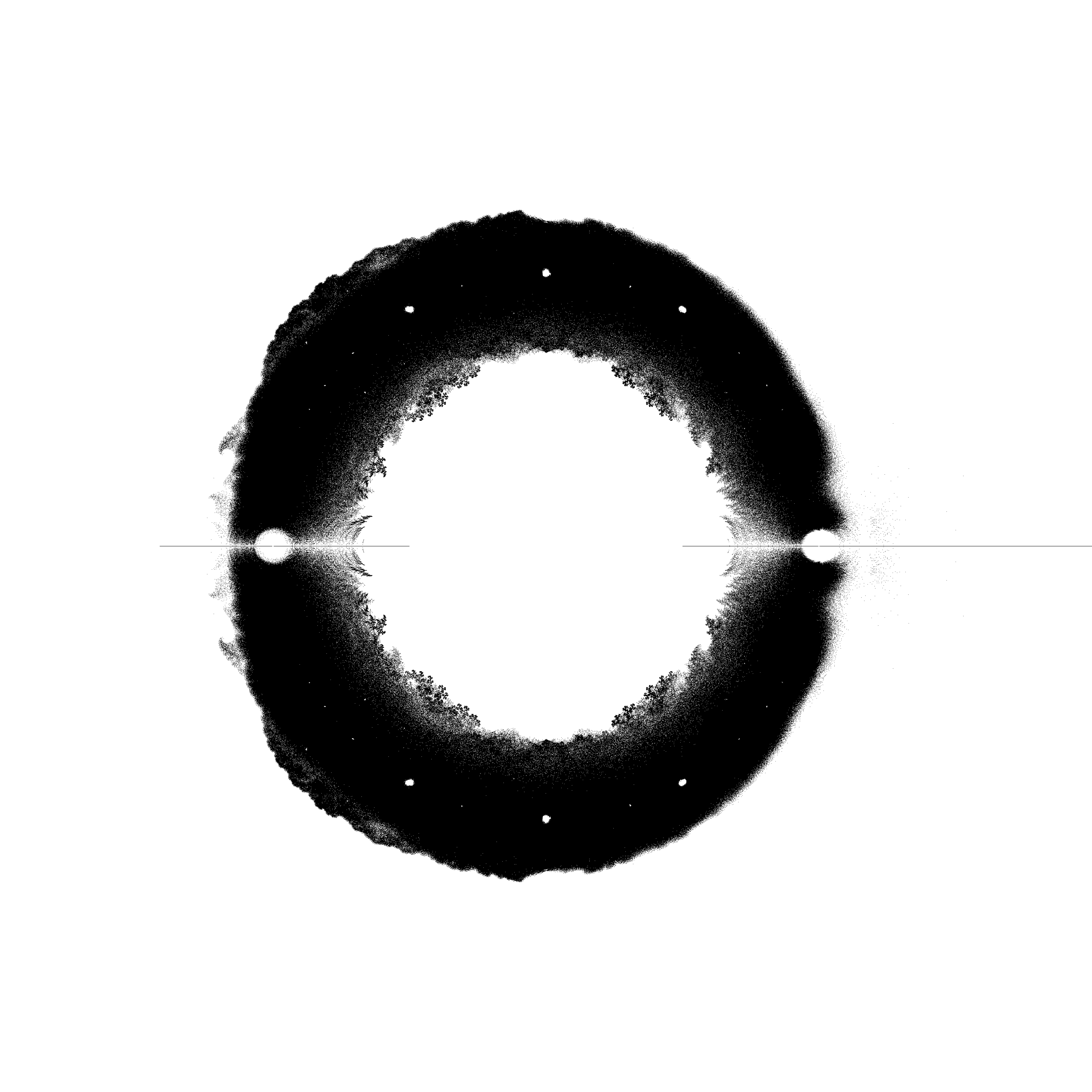}
\caption{An approximation of the preperiodic Thurston set, $\Omega_2^{pcf}$, consisting of the roots of all minimal polynomials associated to postcritically finite tent maps for which the sum of the pre-critical length and the period is at most 22. This set is shown in red in Figure \ref{fig:overlay}. Compare this with the Thurston set $\Omega_2^{cp}$ in Figure \ref{fig:thurston_set}, and note in particular the difference in a large neighborhood of the point 1.
}
\label{fig:preperiodic}
\end{center}
\end{figure}

As outlined in section \S \ref{ss:IFSdescription}, a point $z \in \mathbb{D}$ is in $\Omega_2^{cp}$ if and only if $0$ is in the limit set of the iterated function system generated by $f_z,g_z$, where
 $$f_z:x \mapsto zx+1, \quad g_z:x \mapsto zx-1.$$ 
Denote the alphabet $\{f_z,g_z\}$ by $\mathcal{F}_z$ and denote the alphabet of inverses $\{f^{-1}_z,g^{-1}_z\}$ by $\mathcal{F}^{-1}_z$.
 For a word $w=w_1,\dots,w_n$ in the alphabet $\mathcal{F}_z$ or in the alphabet $\mathcal{F}^{-1}_z$, define the action of $w$ on $\mathbb{C}$ by 
$$w(x) = w_n \circ \dots \circ w_1 (x).$$

\begin{lemma} \label{l:notinThurstonSetCriterion}
Fix $z \in \mathbb{D} \setminus \{0\}$.  If there exists $n \in \mathbb{N}$ such that 

$$\min \left \{ |v(0)| : v \in (\mathcal{F}^{-1}_z)^n \right \} > \frac{1}{1-|z|},$$
then $z \not \in \Omega_2^{cp}$.
\end{lemma}

\begin{proof}
Suppose $z \in \mathbb{D} \cap \Omega_2^{cp}$.  Then $0$ is in the limit set $\Lambda_z$.  
Since $\Lambda_z = f_z(\Lambda_z) \cup g_z(\Lambda_z)$, it follows that $\Lambda_z$ is fixed by taking the union of the images of $\Lambda_z$ under all words of length $n$, for any $n \in \mathbb{N}$:
$$\Lambda_z = \bigcup_{w \in (\mathcal{F}_z)^n} w(\Lambda_z).$$ Hence, for
any $n \in \mathbb{N}$, each point in $\Lambda_z$ is the image of a point
$\Lambda_z$ under some word in $\mathcal{F}_z$ of length $n$.  In
particular, $0$ is the the image of a point in $\Lambda_z$ under some word
in $\mathcal{F}_z$ 
of length $n$.  Since $\Lambda_z \subset B_{\frac{1}{1-|z|}}(0)$ by Lemma \ref{l:boundsonlimitset}, this implies that for any $n \in \mathbb{N}$,
$$ \left( \bigcup_{v \in (\mathcal{F}^{-1}_z)^n} v(0) \right) \cap B_{\frac{1}{1-|z|}}(0) \neq \emptyset.$$

\end{proof}

\begin{proof}[Proof of the Two Thurston Sets Theorem
  \ref{mainthm:prepernotequal}]
  Let $\beta$ be the leading root of the polynomial
  \[P(x)=x^{12} - 2x^{11} + x^{10} - 2x^9 + x^8 - 2x^7 + 2x^6 - 2x^5 + 4x^4 - 2x^3 + 4x^2 - 4x + 2.\]
  (The value of $\beta$ is approximately $1.94848$.) 
  By computation, the minimal $\beta$-itinerary is
  $$w=1000011100(101000)^\infty.$$

  Because $P$ is irreducible, any roots of $P$ must be in $\Omega^{pcf}$. 
Let $p$ be the root of $P$ with approximate value  
$$p \approx 0.5393738531461442 + 0.4050155839374199i.$$
Since $|p|$ is approximately $0.674509$, $p \in \mathbb{D} \cap \Omega_2^{pcf}$.  

 Let $\mathcal{F}^{-1}_p$ be the alphabet consisting of the two maps $f_p^{-1}$ and $g_p^{-1}$, where
$$f_p^{-1}:x \mapsto \frac{x-1}{p}, \quad g_p^{-1}:x \mapsto \frac{x+1}{p}.$$
Computation shows that
$$\min \left \{ |v(0)| : v \in (\mathcal{F}^{-1}_p)^5 \right \} \approx 4.3792,$$
which is much bigger than $\frac{1}{1-|p|} \approx 3.07228$. Consequently, Lemma \ref{l:notinThurstonSetCriterion}  implies that $p \not \in \Omega_2^{cp}$.

\end{proof}

\bibliographystyle{alpha}
\bibliography{references}

\noindent Harrison Bray, University of Michigan, Department of Mathematics, 530 Church Street, Ann Arbor MI, \mbox{\url{hbray@umich.edu}}

\noindent Diana Davis, Swarthmore College, Department of Mathematics and Statistics, 500 College Avenue, Swarthmore PA, \mbox{\url{dianajdavis@gmail.com}}

\noindent Kathryn Lindsey, Boston College,  Department of Mathematics, Maloney Hall, Fifth Floor, Chestnut Hill, MA, \mbox{\url{kathryn.a.lindsey@gmail.com}}

\noindent Chenxi Wu, Rutgers University, Department of Mathematics, 110 Frelinghuysen Road, Piscataway, NJ, \mbox{\url{wuchenxi2013@gmail.com}}

\end{document}